\documentclass[12pt]{amsart}
\usepackage{amsmath, amssymb, amscd, amsbsy, amsthm,braket}
\usepackage{graphicx}
\newtheorem{theorem}{Theorem}[section]
\newtheorem{lemma}[theorem]{Lemma}
\newtheorem{prop}[theorem]{Proposition}
\newtheorem{cor}[theorem]{Corollary}
\theoremstyle{definition}

\theoremstyle{remark}
\newtheorem{remark}[theorem]{Remark}
\numberwithin{equation}{section}

\def\Im{\operatorname{Im}}
\def\Ker{\operatorname{Ker}}
\def\Aut{\operatorname{Aut}}
\def\GL{\operatorname{GL}}

\def\Coker{\operatorname{Coker}}
\def\Int{\operatorname{Int}}

\def\Hom{\operatorname{Hom}}
\def\Out{\operatorname{Out}}

\def\Z{{\mathbb Z}}

\newcommand\hb[1]{\mathcal{H}_{#1}}
\begin{document}

\title[A twisted first homology of the handlebody MCG]{A twisted first homology group of the handlebody mapping class group}
\author{Tomohiko Ishida}
\address{Department of Mathematics, 
Kyoto University, 
Kitashirakawa Oiwake-cho, Sakyo-ku, Kyoto 606-8502, Japan.}
\email{ishidat@math.kyoto-u.ac.jp}

\author{Masatoshi Sato}
\address{Department of Mathematics Education, Faculty of Education,
Gifu University, 1-1 Yanagido, Gifu City, Gifu 501-1193, Japan}
\email{msato@gifu-u.ac.jp}

\date{\today}

\begin{abstract}
Let $H_g$ be a 3-dimensional handlebody of genus $g$.
We determine the twisted first homology group of the mapping class group of $H_g$ 
with coefficients in the first integral homology group of the boundary surface $\partial H_g$ for $g\ge2$.
\end{abstract}

\maketitle
\section{Introduction}
Let $H_g$ be a 3-dimensional handlebody of genus $g$,
and $\Sigma_g$ the boundary surface $\partial H_g$.
We denote by $\mathcal{H}_g$ and $\mathcal{M}_g$
the mapping class group of $H_g$ and the boundary surface $\Sigma_{g}$, respectively.
These are the groups of isotopy classes of orientation preserving homeomorphisms of $\Sigma_g$ and $H_g$.
Let $D$ be a closed 2-disk in the boundary $\Sigma_g$ of the handlebody,
and pick a point $*$ in $\Int D$.
Let us denote by $\mathcal{H}_g^*$ and $\mathcal{H}_{g,1}$
the groups of the isotopy classes of orientation preserving homeomorphisms of $H_g$
fixing $*$ and $D$ pointwise, respectively.
We also denote by $\mathcal{M}_g^*$ and $\mathcal{M}_{g,1}$
the groups of the isotopy classes of orientation preserving homeomorphisms of $\Sigma_g$
fixing $*$ and $D$ pointwise, respectively.
We use integral coefficients for homology groups unless specified throughout the paper.

In the cases of the mapping class group  $\mathcal{M}_g^*$ and $\mathcal{M}_g$ of a surface $\Sigma_g$,
Morita~\cite[Corollary~5.4]{morita89} determined the first homology group
with coefficients in the first integral homology group of the surface.
Morita~\cite[Remark~6.3]{morita93} extended
the first Johnson homomorphism to a crossed homomorphism 
$\mathcal{M}_g^*\to \frac{1}{2}\Lambda^3(H_1(\Sigma_g))$,
and showed that the contraction
of this crossed homomorphism gives isomorphisms
$H_1(\mathcal{M}_g^*;H_1(\Sigma_g))\cong \mathbb{Z}$
and $H_1(\mathcal{M}_g;H_1(\Sigma_g))\cong \mathbb{Z}/(2g-2)\mathbb{Z}$ when $g\ge2$.
For twisted homology groups of the mapping class groups of nonorientable surfaces,
see Stukow~\cite{stukow15}.
In the cases of the automorphism group $\Aut F_n$
and the outer automorphism group $\Out F_n$ of a free group $F_n$ of rank $n$,
Satoh~\cite{satoh06} computed
$H_1(\Aut F_n;H^1(F_n))$ and $H_1(\Out F_n;H^1(F_n))$ for $n\ge2$.
Kawazumi~\cite{kawazumi05} extended the first Andreadakis-Johnson homomorphism to
a crossed homomorphism $\Aut F_n\to H^1(F_n)\otimes H_1(F_n)^{\otimes 2}$.
The contraction of this crossed homomorphism also gives isomorphisms
$H_1(\Aut F_n;H^1(F_n))\cong \mathbb{Z}$ and
$H_1(\Out F_n;H^1(F_n))\cong \mathbb{Z}/(n-1)\mathbb{Z}$.

In this paper,
we compute the twisted first homology groups of $\mathcal{H}_g$ and $\mathcal{H}_g^*$
with coefficients in the first integral homology group of the boundary surface $\Sigma_g$.
Note that the restrictions of homeomorphisms of $H_g$ to  $\Sigma_g$ induce
an injective homomorphism $\mathcal{H}_g\to\mathcal{M}_g$,
and we treat the group $\mathcal{H}_g$ as a subgroup of $\mathcal{M}_g$.
The followings are main theorems in this paper.
\begin{theorem}\label{thm:main theorem}
\[
H_1(\mathcal{H}_g;H_1(\Sigma_g))\cong
\begin{cases}
\mathbb{Z}/{(2g-2)}\Z &\text{ if }g\ge4,\\
\mathbb{Z}/4\Z\oplus\mathbb{Z}/2\Z &\text{ if }g=3,\\
(\mathbb{Z}/2\Z )^2&\text{ if }g=2,
\end{cases}
\]
Furthermore, when $g\ge4$,
the homomorphism
$H_1(\mathcal{H}_g;H_1(\Sigma_g))\to H_1(\mathcal{M}_g;H_1(\Sigma_g))$ 
induced by the inclusion is an isomorphism.
When $g=2,3$,
this homomorphism is surjective 
and the kernel is isomorphic to $\mathbb{Z}/2\mathbb{Z}$.
\end{theorem}

\begin{theorem}\label{thm:main theorem2}
\item
\[
H_1(\mathcal{H}_{g,1};H_1(\Sigma_g))\cong H_1(\mathcal{H}_g^*;H_1(\Sigma_g))\cong
\begin{cases}
\mathbb{Z} &\text{if }g\ge4,\\
\mathbb{Z}\oplus\mathbb{Z}/2\mathbb{Z} &\text{if }g=2,3.
\end{cases}
\]
Furthermore, when $g\ge4$,
the homomorphism
$H_1(\mathcal{H}_g^*;H_1(\Sigma_g))\to H_1(\mathcal{M}_g^*;H_1(\Sigma_g))$ 
induced by the inclusion is an isomorphism.
When $g=2,3$,
this homomorphism is surjective 
and the kernel is isomorphic to $\mathbb{Z}/2\mathbb{Z}$.
\end{theorem}

In this paper, we also study relationships between the second homology groups of 
$\mathcal{H}_g$, $\mathcal{H}_g^*$, and $\mathcal{H}_{g,1}$.
The second homology group of $\mathcal{M}_g$ is calculated
by Harer~\cite{harer83} when $g\ge5$.
It contains some minor mistakes and these are corrected in \cite{harer85} later.
For surfaces with an arbitrary number of punctures and boundary components,
see Korkmaz-Stipsicz~\cite{korkmaz03}.
See also Benson-Cohen~\cite{benson91} and Sakasai~\cite{sakasai12} for low genera.
There are some results which imply that
the cohomology group of the handlebody mapping class group $\mathcal{H}_g$
is similar to that of $\mathcal{M}_g$.
Morita~\cite[Proposition~3.1]{morita87} showed that
the rational cohomology group of any subgroup of the mapping class group decomposes into a direct sum.
Later, Kawazumi-Morita~\cite[Proposition~5.2]{kawazumi01} generalized it 
to the cohomology group with coefficients in $A=\mathbb{Z}[1/(2g-2)]$.
In particular, the cohomology group of the handlebody mapping class group with a puncture decomposes as
\[
H^n(\mathcal{H}_g^*;A)\cong H^n(\mathcal{H}_g;A)\oplus H^{n-1}(\mathcal{H}_g;H^1(\Sigma_g;A))\oplus H^{n-2}(\mathcal{H}_g;A).
\]
Hatcher-Wahl~\cite{hatcher10} showed that
the integral cohomology groups of the mapping class groups of 3-manifolds stabilize in more general settings.
Hatcher also announced that the rational stable cohomology group 
coincides with the polynomial ring generated by the even Morita-Mumford classes.
However, as far as we know,
even the second integral homology group of handlebody mapping class groups 
has not been computed yet. 

Here is the outline of our paper:

In Section 2,
we investigate the relationship between the second integral homology group
of the handlebody mapping class group 
fixing a point or a 2-disk in $\Sigma_g$ pointwise 
with that of $\mathcal{H}_g$ using Theorem~\ref{thm:main theorem}.

In Section 3,
we compute the twisted first homology group $H_1(\mathcal{H}_g;H_1(\Sigma_g))$
to prove Theorem~\ref{thm:main theorem} in the case when $g\ge4$.
We also compute the twisted first homology groups of $\mathcal{H}_g$ with coefficients in $\Ker(H_1(\Sigma_g)\to H_1(H_g))$ and $H_1(H_g)$.

Let $\mathcal{L}_g$ denote the kernel of the homomorphism $\mathcal{H}_g\to \Out(\pi_1H_g)$.
The exact sequence
\[
\begin{CD}
1@>>>\mathcal{L}_g@>>>\mathcal{H}_g@>>>\Out(\pi_1H_g)@>>>1
\end{CD}
\]
induces exact sequences between their first homology groups with coefficients in $\Ker(H_1(\Sigma_g)\to H_1(H_g))$ and $H_1(H_g)$.
Luft~\cite{luft78}  showed that the group $\mathcal{L}_g$ coincides with the twist group,
which is generated by Dehn twists along meridian disks.
Satoh~\cite{satoh06} determined the twisted first homology groups $H_1(\Out F_n;H_1(F_n))$ and $H_1(\Out F_n;H^1(F_n))$.
Applying Luft's and Satoh's results to the exact sequences, 
we can determine $H_1(\mathcal{H}_g;H_1(\Sigma_g))$ when $g\ge4$.

In Section 4,
we review a finite presentation of the handlebody mapping class group $\mathcal{H}_g$ 
given by Wajnryb~\cite{wajnryb98}.

In Section 5,
we compute the twisted first homology group $H_1(\mathcal{H}_g;H_1(\Sigma_g))$,
using the Wajnryb's presentation of the handlebody mapping class group $\mathcal{H}_g$ 
to prove Theorem~\ref{thm:main theorem} in the case when $g=2,3$.

In Section 6,
we prove Theorem~\ref{thm:main theorem2} 
and also compute the twisted first homology groups of $\mathcal{H}_g^*$
with coefficients in $\Ker(H_1(\Sigma_g)\to H_1(H_g))$ and $H_1(H_g)$.

\section{On the second homology of the handlebody mapping class groups fixing a point or a 2-disk pointwise}
In this section, we introduce some corollaries of Theorem~\ref{thm:main theorem}
which give relationships between the second homology groups of  $\mathcal{H}_g$, $\mathcal{H}_g^*$ and $\mathcal{H}_{g,1}$.

\begin{figure}[htbp]
\begin{center}
\includegraphics[width=90mm]{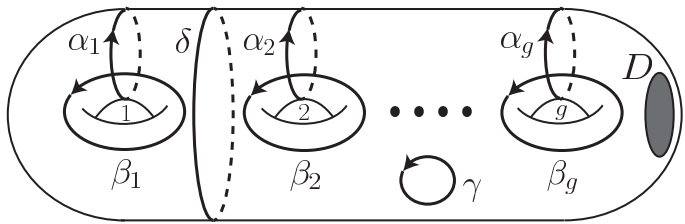}
\end{center}
\caption{a 2-disk $D$ and simple closed curves $\alpha_1,\ldots, \alpha_g, \beta_1,\ldots, \beta_g, \gamma$}
\label{fig:scc1}
\end{figure}
Let $U\Sigma_g$ denote the unit tangent bundle of $\Sigma_g$.
Let $\alpha_1,\ldots, \alpha_g, \beta_1, \ldots, \beta_g$ be oriented smooth simple closed curves as in figure~\ref{fig:scc1},
and denote their homology classes in $H_1(\Sigma_g)$
by $x_1=[\alpha_1], x_2=[\alpha_2], \ldots, x_g=[\alpha_g], y_1=[\beta_1], y_2=[\beta_2], \ldots, y_g=[\beta_g]$.
We also denote by $\gamma$ a null-homotopic smooth simple closed curve in figure~\ref{fig:scc1}.
There are natural liftings of $\alpha_1, \ldots, \alpha_g, \beta_1, \ldots, \beta_g, \gamma$ to $U\Sigma_g$,
and let us denote their homology classes in $H_1(U\Sigma_g)$
by $\tilde{x}_1, \tilde{x}_2, \ldots, \tilde{x}_g$, $\tilde{y}_1, \tilde{y}_2, \ldots, \tilde{y}_g$, $z$, respectively.
For a group $G$ and a $G$-module $M$,
let us denote by $M_{G}$ its coinvariant, that is,
the quotient of $M$ by the submodule spanned by the set $\{gm-m\,|\,m\in M, g\in G\}$.
\begin{lemma}\label{lem:unit-tan}
For $g\ge2$,
\[
H_1(U\Sigma_g)_{\mathcal{H}_g}=0.
\]
\end{lemma}
\begin{proof}
For a simple closed curve $c$ in $\Sigma_g$,
we denote by $t_c$ the Dehn twist along $c$.
As in \cite[Theorem~1B]{johnson80}, 
we have $t_{\alpha_i}(\tilde{y}_i)=\tilde{y}_i+\tilde{x}_i$ for $i=1,\ldots,g$.
Note that our $\tilde{c}$ is denoted by $\vec{c}$ in \cite{johnson80},
and is different from $\tilde{c}$.
Hence, we have $\tilde{x}_1=\cdots=\tilde{x}_g=0\in H_1(U\Sigma_g)_{\mathcal{H}_g}$.
Let $\delta'_i$ and $\alpha'_i$ be simple closed curves
as depicted in Figure~\ref{fig:scc2} for $1\le i\le g-1$.
Let us denote $h_i=t_{\delta_i'}^{-1}t_{\beta_i}t_{\alpha_{i+1}}\in\mathcal{M}_g$.
Since $h_i(\alpha_l)=\alpha_l$ when $l\ne i$, and $h_i(\alpha_i)=\alpha'_i$,
the mapping class $h_i$ is actually an element of the handlebody mapping class group $\mathcal{H}_g$.
\begin{figure}[htbp]
\begin{center}
\includegraphics[width=100mm]{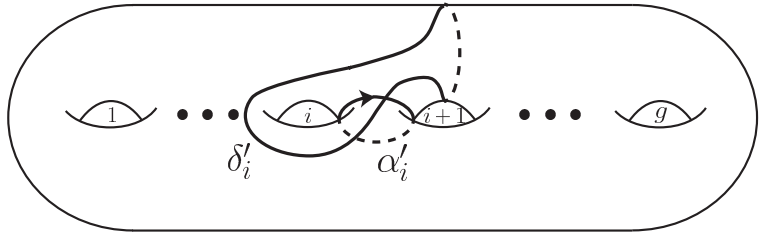}
\end{center}
\caption{simple closed curves $\delta'$ and $\alpha'_i$}
\label{fig:scc2}
\end{figure}
We obtain
\[
h_i(\tilde{x}_i)=\tilde{x}_i-\tilde{x}_{i+1}-z \text{ and }
h_i(\tilde{y}_{i+1})=\tilde{y}_i+\tilde{y}_{i+1}-z,
\]
for $i=1,\ldots, g-1$. 
Thus we have
$z=\tilde{y}_1=\cdots=\tilde{y}_{g-1}=0\in H_1(U\Sigma_g)_{\mathcal{H}_g}$.
Since the rotation $r\in \mathcal{H}_g$ of the surface $\Sigma_g$
about a vertical line by $180$~degrees maps $\tilde{y}_g$ to $-\tilde{y}_1$,
we also obtain $\tilde{y}_g=0$.
\end{proof}
Using the mapping classes $h_i$ and $r$,
we can also show:
\begin{lemma}\label{lem:H_0-L}
For $g\ge2$,
\[
\Ker(H_1(\Sigma_g)\to H_1(H_g))_{\mathcal{H}_g}=0.
\]
\end{lemma}

\begin{prop}\label{prop:pointed}
When $g\ge4$,
\[
H_2(\mathcal{H}_g^*)\cong H_2(\mathcal{H}_g)\oplus\mathbb{Z}.
\]
\end{prop}
\begin{proof}
Let us denote the Lyndon-Hochschild-Serre spectral sequences of the forgetful exact sequences
\[
\begin{CD}
1@>>>\pi_1\Sigma_g@>>> \mathcal{M}_g^* @>>>\mathcal{M}_g@>>>1,\\
1@>>>\pi_1\Sigma_g@>>> \mathcal{H}_g^* @>>>\mathcal{H}_g@>>>1
\end{CD}
\]
by $\{E^r_{p,q}\}$ and $\{\bar{E}^r_{p,q}\}$, respectively.
By Lemma~\ref{lem:unit-tan},
we have $H_1(\Sigma_g)_{\mathcal{M}_g}=H_1(\Sigma_g)_{\mathcal{H}_g}=0$.
Thus the $E^{\infty}$ terms of both spectral sequences are as follows.
\[
\begin{array}{|c|c|c|}
\hline
E_{0,2}^\infty&*&*\\
\hline
0&E_{1,1}^{\infty}&*\\
\hline
\mathbb{Z}&H_1(\mathcal{M}_g)&H_2(\mathcal{M}_g)\\
\hline
\end{array}\quad\quad
\begin{array}{|c|c|c|}
\hline
\bar{E}_{0,2}^\infty&*&*\\
\hline
0&\bar{E}_{1,1}^{\infty}&*\\
\hline
\mathbb{Z}&H_1(\mathcal{H}_g)&H_2(\mathcal{H}_g)\\
\hline
\end{array}
\]
Therefore, we have a morphism of exact sequences
\[
\begin{CD}
0@>>>\bar{E}_{0,2}^\infty@>>> \Ker(H_2(\mathcal{H}_g^*)\to H_2(\mathcal{H}_g))@>>>\bar{E}_{1,1}^{\infty}@>>>0\\
@.@VVV@VVV@VVV@.\\
0@>>>E_{0,2}^\infty@>>> \Ker(H_2(\mathcal{M}_g^*)\to H_2(\mathcal{M}_g))@>>>E_{1,1}^{\infty}@>>>0
\end{CD}
\]
induced by the inclusion $\mathcal{H}_g^*\to\mathcal{M}_g^*$.
As explained in \cite[Propositions~1.4 and 1.5]{korkmaz03},
$\Ker(H_2(\mathcal{M}_g^*)\to H_2(\mathcal{M}_g))\cong\mathbb{Z}$ 
and $E_{0,2}^\infty=E_{0,2}^2\cong\mathbb{Z}$ when $g\ge4$.
It is also true when $g=3$ as in \cite[Corollary~4.9]{sakasai12} (see also \cite{pitsch13}).
Moreover, there exists a surjective homomorphism 
$S_1:H_2(\mathcal{M}_g^*)\to \mathbb{Z}$ defined in \cite[Section~0]{harer83}
which maps the fundamental class $[\Sigma_g]\in H_2(\Sigma_g)=E_{0,2}^\infty$ to $(2g-2)$-times a generator
and whose restriction to $\Ker(H_2(\mathcal{M}_g^*)\to H_2(\mathcal{M}_g))$ is surjective.
These facts show that $E_{1,1}^\infty$ is a cyclic group of order $2g-2$.
Since Morita \cite{morita89} showed $E^2_{1,1}=H_1(\mathcal{M}_g;H_1(\Sigma_g))\cong\mathbb{Z}/(2g-2)\mathbb{Z}$ when $g\ge2$,
we obtain $E_{1,1}^2=E_{1,1}^\infty$.

When $g\ge4$, 
this fact and the isomorphism 
$H_1(\mathcal{H}_g;H_1(\Sigma_g))\cong H_1(\mathcal{M}_g;H_1(\Sigma_g))$
show that in the commutative diagram
\[
\begin{CD}
\bar{E}^2_{1,1}@>>>\bar{E}^\infty_{1,1}\\
@VVV@VVV\\
E^2_{1,1}@>>>E^\infty_{1,1}\\
\end{CD}
\]
we have an isomorphism $\bar{E}_{1,1}^\infty\cong E_{1,1}^\infty$.
As a conclusion,
we obtain 
\[
\Ker(H_2(\mathcal{H}_g^*)\to H_2(\mathcal{H}_g))\cong\Ker(H_2(\mathcal{M}_g^*)\to H_2(\mathcal{M}_g))\cong\mathbb{Z}.
\]
Consider the commutative diagram
\[
\begin{CD}
0@>>>\mathbb{Z}@>>>H_2(\mathcal{H}_g^*)@>>>H_2(\mathcal{H}_g)@>>>0,\\
@.@|@VVV@VVV@.\\
0@>>>\mathbb{Z}@>>>H_2(\mathcal{M}_g^*)@>>>H_2(\mathcal{M}_g)@>>>0.
\end{CD}
\]
Since the lower exact sequence splits,
we obtain $H_2(\mathcal{H}_g^*)\cong H_2(\mathcal{H}_g)\oplus\mathbb{Z}$.
\end{proof}

When $g\ge2$, Lemma~\ref{lem:unit-tan} and the five term exact sequences
induced by the exact sequences
\[
\begin{CD}
1@>>>\pi_1\Sigma_g@>>> \mathcal{H}_g^* @>>>\mathcal{H}_g@>>>1,\\
1@>>>\pi_1U\Sigma_g@>>>\mathcal{H}_{g,1}@>>>\mathcal{H}_g@>>>1
\end{CD}
\]
imply:
\begin{lemma}\label{lem:abel-pointed}
When $g\ge2$,
\[
H_1(\mathcal{H}_{g,1})\cong H_1(\mathcal{H}_g^*)\cong H_1(\mathcal{H}_g).
\]
\end{lemma}

\begin{remark}\label{rem:abelianization}
By the Wajnryb's presentation which we review in Section~\ref{subsection:wajnryb},
we can compute the abelianization as follows:
\[
H_1(\mathcal{H}_g)\cong
\begin{cases}
\mathbb{Z}\oplus\mathbb{Z}/2\Z &\text{if }g=1,\\
\mathbb{Z}\oplus (\mathbb{Z}/2\Z )^2&\text{if }g=2,\\
\mathbb{Z}/2\Z &\text{if }g\ge3.
\end{cases}
\]
We can also see that it is generated by $s_1=t_{\beta_1}t_{\alpha_1}^2t_{\beta_1}$ when $g\ge 3$.
Note that Wajnryb made a mistake in his calculation of the abelianization in \cite[Theorem~20]{wajnryb98} when $g=2$.
\end{remark}

In the following, 
we choose a 2-disk $D$ in the boundary $\Sigma_g$
so that it is disjoint from the simple closed curves
$\alpha_1,\ldots, \alpha_g, \beta_1, \ldots, \beta_g$
as in Figure \ref{fig:scc1} 
and pick a point $*$ in $\Int D$.
\begin{lemma}\label{lem:boundary}
When $g\ge3$,
\[
H_2(\mathcal{H}_g^*)\cong H_2(\mathcal{H}_{g,1})\oplus \mathbb{Z}.
\]
\end{lemma}
\begin{proof}
Let $\pi:\mathcal{H}_{g,1}\to\mathcal{H}_g^*$ denote the forgetful map.
The Gysin exact sequence of the central extension
\[
\begin{CD}
0@>>>\mathbb{Z}@>>> \mathcal{H}_{g,1} @>\pi>>\mathcal{H}_g^*@>>>1.
\end{CD}
\]
is written as
\[
H_1(\mathcal{H}_g^*)\xrightarrow{\pi^!} H_2(\mathcal{H}_{g,1})\to H_2(\mathcal{H}_g^*)\to \mathbb{Z}\xrightarrow{\pi^!}  H_1(\mathcal{H}_{g,1}).
\]
Recall that the Gysin homomorphism $\pi^!: H_1(\mathcal{H}_g^*)\to H_2(\mathcal{H}_{g,1})$
maps $[h]$ to $[\tilde{h}|t_{\partial D}]-[t_{\partial D}|\tilde{h}]$ in the bar resolution for $g\in\mathcal{H}_g^*$,
where $\tilde{h}\in \mathcal{H}_{g,1}$ is the inverse image of $h$ under $\pi$.
By Lemma~\ref{lem:abel-pointed} and \cite[Theorem~20]{wajnryb98},
$H_1(\mathcal{H}_g^*)$ is the cyclic group of order $2$ generated by $s_1$ when $g\ge3$.
Note that we can choose a representating diffeomorphism of $s_1$
whose support is in a genus 1 subsurface $S$ of $\Sigma_g-\Int D$.
Moreover, using the lantern relation,
we can obtain a 2-chain which bounds $[t_{\partial D}]\in C_1(\mathcal{H}_{g,1})$ whose support is in $(\Sigma_g-\Int D)-S$.
Thus there exists a 3-chain which bounds $[\tilde{s}_1|t_{\partial D}]-[t_{\partial D}|\tilde{s}_1]\in C_2(\mathcal{H}_{g,1})$,
where $\tilde{s}_1=t_{\beta_1}t_{\alpha_1}^2t_{\beta_1}\in \mathcal{H}_{g,1}$.
Hence, the Gysin homomorphism $\pi^!: H_1(\mathcal{H}_g^*)\to H_2(\mathcal{H}_{g,1})$ is the zero map.
Since $[t_{\partial D}]=0\in H_1(\mathcal{H}_{g,1})$, 
the homomorphism 
$\pi^!:\mathbb{Z}\to H_1(\mathcal{H}_{g,1})$ is also trivial. 
Thus we obtain the exact sequence
\[
0\to H_2(\mathcal{H}_{g,1})\to H_2(\mathcal{H}_g^*)\to \mathbb{Z}\to 0.
\]
\end{proof}

Both of the direct sum decompositions of $H_2(\mathcal{H}_g^*)$
in Proposition~\ref{prop:pointed} and Lemma~\ref{lem:boundary}
are induced by the composition of
the natural homomorphism $H_2(\mathcal{H}_g^*)\to H_2(\mathcal{M}_g^*)$ and
$S_1:H_2(\mathcal{M}_g^*)\to \mathbb{Z}$ defined in \cite[Section 4]{harer83} up to sign.
Thus we obtain:
\begin{cor}
When $g\ge4$,
\[
H_2(\mathcal{H}_{g,1})\cong H_2(\mathcal{H}_g).
\]
\end{cor}

\section{Proof of Theorem~\ref{thm:main theorem} for $g\ge4$}
In the rest of this paper,
we write $H$ for $H_1(\Sigma_g)$
and denote by $L$ the kernel of the homomorphism
$H_1(\Sigma_g)\to H_1(H_g)$ induced by the inclusion 
for simplicity.
Note that $H_1(H_g)$ is isomorphic to $H/L$ as an $\mathcal{H}_g$-module.
In this section,
we prove Theorem~\ref{thm:main theorem} when $g\ge4$.
Luft's result on $\Ker(\mathcal{H}_g\to\Out F_g)$ and Satoh's result on $\Out F_g$
make it much easier to determine the first homology $H_1(\mathcal{H}_g;H)$ when $g\ge4$
than when $g=2,3$.

\begin{lemma}\label{lem:morita-homo}
Let $g\ge2$, and $G$ a subgroup of the mapping class group $\mathcal{M}_g$.
When the induced map $H_1(U\Sigma_g)_G\to H_G$ by the natural projection is injective,
there exists a surjective homomorphism
\[
H_1(G;H)\to \mathbb{Z}/(2g-2)\mathbb{Z}.
\]
\end{lemma}
\begin{proof}
The exact sequence
\[
0\to \mathbb{Z}/(2g-2)\mathbb{Z}\to H_1(U\Sigma_g)\to H\to0
\]
induces the exact sequence
\[
H_1(G;H)\to \mathbb{Z}/(2g-2)\mathbb{Z}\to H_1(U\Sigma_g)_{G}\to H_{G}.
\]
Thus, 
we obtain the surjective homomorphism $H_1(\mathcal{H}_g;H)\to \mathbb{Z}/(2g-2)\Z $.
\end{proof}
\begin{remark}\label{rem:surjective-homo}
The homomorphism $H_1(G;H)\to \mathbb{Z}/(2g-2)\mathbb{Z}$ is written in \cite[Section~6]{morita89} explicitly.
This coincide with
the mod $(2g-2)$-reduction of the contraction of the twisted homomorphism
called the first Johnson homomorphism.
In particular, the homomorphism
$H_1(G;H)\to H_1(\mathcal{M}_g;H)$
induced by the inclusion is surjective.
Note that the handlebody mapping class group $\hb{g}$ satisfies the assumption of Lemma~\ref{lem:morita-homo} because of Lemma~\ref{lem:unit-tan}.
\end{remark}

By Lemma~\ref{lem:unit-tan}, we obtain a lower bound on the order of $H_1(\mathcal{H}_g;H)$.
For a simple closed curve $c$ in $\Sigma_g$,
we denote by $\mathcal{H}_g(c)$ the subgroup of $\mathcal{H}_g$ 
which preserves the curve $c$ setwise.
\begin{lemma}\label{lem;lg}
Let $M$ be an $\mathcal{H}_g$-module on which $\mathcal{L}_g$ acts trivially.
Then, we have an exact sequence
\[
\begin{CD}
M_{\mathcal{H}_g(\alpha_1)}@>>> H_1(\mathcal{H}_g;M)@>>> H_1(\Out F_g; M_{\mathcal{L}_g})@>>>0.
\end{CD}
\]
\end{lemma}
\begin{proof}
The short exact sequence $1\to \mathcal{L}_g\to \mathcal{H}_g\to \Out F_g\to 1$ induces an exact sequence
\begin{equation}\label{eq:exact lg}
\begin{CD}
H_1(\mathcal{L}_g;M)_{\mathcal{H}_g}\to H_1(\mathcal{H}_g;M)\to H_1(\Out F_g;M_{\mathcal{L}_g})\to 0.
\end{CD}
\end{equation}
Luft~\cite[Corollary~2.4]{luft78} proved that $\mathcal{L}_g$ is
normally generated by the Dehn twists along the curves $\alpha_1$ and $\delta$ in Figure~\ref{fig:scc1}.
When $g\ge2$, the lantern relation implies that the Dehn twist $t_\delta$
can be written as a product of Dehn twists along boundary curves of meridian disks.
Thus, $\mathcal{L}_g$ is normally generated by the Dehn twist along $\alpha_1$.
Since $\mathcal{L}_g$ acts on $M$ trivially,
we have $H_1(\mathcal{L}_g;M)_{\mathcal{H}_g}=(H_1(\mathcal{L}_g)\otimes M)_{\mathcal{H}_g}$,
and it is generated by $\{t_{\alpha_1}\otimes m\,|\,m\in M\}$.
Since the surjective homomorphism $M\to H_1(\mathcal{L}_g;M)_{\mathcal{H}_g}$
defined by $m\mapsto t_{\alpha_1}\otimes m$ factors through $M_{\mathcal{H}_g(\alpha_1)}$,
the exact sequence~(\ref{eq:exact lg}) and this homomorphism induce the desired exact sequence.
\end{proof}

\begin{lemma}\label{lem:luft-homology}
\begin{enumerate}
\item $H_1(\mathcal{L}_g;H/L)_{\mathcal{H}_g}\cong0\text{ or }\mathbb{Z}/2\mathbb{Z}$ when $g\ge2$.
\item $H_1(\mathcal{L}_g;L)_{\mathcal{H}_g}=0$ when $g\ge3$,
and $H_1(\mathcal{L}_2;L)_{\mathcal{H}_2}\cong0\text{ or }\mathbb{Z}/2\mathbb{Z}$ when $g=2$.
\end{enumerate}
\end{lemma}
\begin{proof}
(1)
Let us denote by $\bar{y}_i$ the image of $y_i$ under the natural homomorphism
$H\to H_1(H_g)\cong H/L$ induced by the inclusion.
There exists a mapping class $r_{1, j}\in \mathcal{H}_g$ for $1<j\le g$
which preserves $\alpha_1$ setwise and satisfies
\[
r_{1, j}(x_{l})=
\begin{cases}
-x_{1}-x_{2}-\cdots-x_j,&\text{ if }l=j,\\
x_l, \text{ otherwise},
\end{cases}\quad
r_{1, j}(y_{l})=
\begin{cases}
y_l-y_j, & \text{ if }1\le l\le j-1,\\
-y_j,&\text{ if }l=j,\\
y_l,&\text{ otherwise}.
\end{cases}
\]
See Lemma~\ref{lem:action} for details.
Then, we have $r_{1,j}(t_{\alpha_1}\otimes \bar{y}_1)=t_{\alpha_1}\otimes \bar{y}_1\in H_1(\mathcal{L}_g;H/L)_{\mathcal{H}_g}$.
Since $r_{1,j}$ commutes with $t_{\alpha_1}$, and $r_{1,j}(\bar{y}_1)=\bar{y}_1-\bar{y}_j$,
we obtain $t_{\alpha_1}\otimes[\bar{y}_j]=0\in H_1(\mathcal{L}_g;H/L)_{\mathcal{H}_g}$ for $j=2,\ldots, g$.
Since the mapping class $(t_{\beta_1}t_{\alpha_1})^3$ 
preserves each $\alpha_i$ setwise for $i=1,2,\ldots, g$,
it is an element in $\mathcal{H}_g$.
Since it satisfies $(t_{\beta_1}t_{\alpha_1})^3(\bar{y}_1)=-\bar{y}_1$,
we have $t_{\alpha_1}\otimes[2\bar{y}_1]=0\in H_1(\mathcal{L}_g;H/L)_{\mathcal{H}_g}$.
As a conclusion, we obtain $H_1(\mathcal{L}_g;H/L)_{\mathcal{H}_g}=0\text{ or }\mathbb{Z}/2\mathbb{Z}$.

(2)
Since $r_{1,j}$ commutes with $t_{\alpha_1}$ for $j=2,3,\ldots, g$,
we obtain $t_{\alpha_1}\otimes[x_1+x_2+\cdots+2x_j]=0\in H_1(\mathcal{L}_g;L)_{\mathcal{H}_g}$.
For $j=1,2,\ldots,g$,
the mapping class $s_j=t_{\beta_j}t_{\alpha_j}^2t_{\beta_j}\in \mathcal{H}_g$
also preserves $\alpha_1$ setwise, and satisfies 
\[
s_j(x_j)=-x_j.
\]
Thus, we also have $t_{\alpha_1}\otimes[2x_j]=0\in H_1(\mathcal{L}_g;L)_{\mathcal{H}_g}$.
Consequently, we obtain $t_{\alpha_1}\otimes[x_1]=t_{\alpha_1}\otimes[x_2]=\cdots=t_{\alpha_1}\otimes[x_{g-1}]=t_{\alpha_1}\otimes[2x_g]=0$,
and it implies $H_1(\mathcal{L}_g;L)_{\mathcal{H}_g}\cong 0$ or $\mathbb{Z}/2\mathbb{Z}$.

Now suppose $g\ge3$.
Then, there exists a mapping class $t_{g-1}$ (see Lemma~\ref{lem:action})
which preserves $\alpha_1$ setwise and satisfies $t_{g-1}(x_g)=x_{g-1}$.
Thus, we also obtain $t_{\alpha_1}\otimes[x_g-x_{g-1}]=0$ when $g\ge3$,
and it implies $H_1(\mathcal{L}_g;L)_{\mathcal{H}_g}=0$.
\end{proof}

Applying Lemma~\ref{lem;lg} to the cases $M=H/L$ and $L$, Lemma~\ref{lem:luft-homology} implies:
\begin{lemma}\label{lem:hg-aut}
When $g\ge3$,
the exact sequence
$1\to\mathcal{L}_g\to \mathcal{H}_g\to \Out F_g\to 1$
induces an isomorphism
\[
H_1(\mathcal{H}_g;L)\cong H_1(\Out F_g;H^1(F_g)).
\]
When $g\ge2$, it also induces an exact sequence
\[
\begin{CD}
\mathbb{Z}/2\mathbb{Z}@>>> H_1(\mathcal{H}_g;H/L) @>>> H_1(\Out F_g;H_1(F_g))@>>>0.
\end{CD}
\]
\end{lemma}

The twisted first homology groups of $\Out F_n$ with coefficients in $H_1(F_n)$ and $H^1(F_n)$
were computed by Satoh~\cite[Theorem~1~(2)]{satoh06} as follows.
\begin{theorem}[Satoh~{\cite[Theorem~1~(2)]{satoh06}}]\label{thm:satoh}
\begin{align*}
H_1(\Out F_n; H^1(F_n))&\cong
\begin{cases}
\mathbb{Z}/(n-1)\Z &\text{when }n\ge4,\\
(\mathbb{Z}/2\Z )^2&\text{when }n=3,\\
\mathbb{Z}/2\Z &\text{when }n=2,
\end{cases}
\\
H_1(\Out F_n; H_1(F_n))&\cong
\begin{cases}
0&\text{when }n\ge4,\\
\mathbb{Z}/2\Z &\text{when }n=2,3.
\end{cases}
\end{align*}
\end{theorem}

By Lemma~\ref{lem:hg-aut} and Theorem~\ref{thm:satoh},
we obtain:
\begin{lemma}\label{lem:first homology-H}
\[
H_1(\mathcal{H}_g;L)\cong
\begin{cases}
\mathbb{Z}/(g-1)\Z &\text{ if }g\ge4,\\
(\mathbb{Z}/2\Z )^2&\text{ if }g=3,
\end{cases}\quad
H_1(\mathcal{H}_g;H/L)\cong \mathbb{Z}/2\Z\text{ if }g\ge4.
\]
\end{lemma}

\begin{remark}
Theorem~\ref{thm:satoh} and Lemma~\ref{lem:first homology-H} show
$\Ker(H_1(\mathcal{H}_g;H/L)\to H_1(\Out F_g;H_1(F_g)))\cong\mathbb{Z}/2\mathbb{Z}$ when $g\ge4$.
Thus Lemma~\ref{lem:luft-homology}~(1) implies
$H_1(\mathcal{L}_g;H/L)_{\mathcal{H}_g}\cong\mathbb{Z}/2\mathbb{Z}$ when $g\ge4$.
\end{remark}

\begin{remark}\label{rem:genus23H/L}
By Lemma~\ref{lem:hg-aut} and Theorem~\ref{thm:satoh},
we see that the order of $H_1(\mathcal{H}_g;H/L)$ for $g=2,3$ is at most 4.
In Propositions~\ref{prop:H/L genus2} and \ref{prop:H/L genus3},
we will show
$H_1(\mathcal{H}_2;L)\cong \mathbb{Z}/2\mathbb{Z}$ and $H_1(\mathcal{H}_g;H/L)\cong(\mathbb{Z}/2\mathbb{Z})^2$ for $g=2,3$.
By Lemma~\ref{lem:luft-homology}~(1) and Theorem~\ref{thm:satoh},
it also follows that $H_1(\mathcal{L}_g;H/L)_{\mathcal{H}_g}\cong\mathbb{Z}/2\mathbb{Z}$
for $g=2,3$.
\end{remark}

By Lemma~\ref{lem:H_0-L},
$H_0(\mathcal{H}_g;L)=L_{\mathcal{H}_g}=0$ for $g\ge2$.
Thus, the short exact sequence of $\mathcal{H}_g$-modules $0\to L\to H\to H/L\to 0$
induces an exact sequence
\begin{equation}\label{eq:l-h/l}
\begin{CD}
H_1(\mathcal{H}_g;L)@>>> H_1(\mathcal{H}_g;H)@>>> H_1(\mathcal{H}_g;H/L)@>>> 0.
\end{CD}
\end{equation}
Lemmas~\ref{lem:morita-homo} and \ref{lem:first homology-H}
and the exact sequence (\ref{eq:l-h/l}) give an upper bound
on the order of $H_1(\mathcal{H}_g;H)$.
Comparing this with the lower bound obtained in Lemma~\ref{lem:morita-homo},
we complete the proof of Theorem~\ref{thm:main theorem} for $g\ge4$.

\begin{remark}\label{rem:L-H}
In the proof of Theorem~\ref{thm:main theorem} above,
we also see the sequence
\[
\begin{CD}
0@>>>H_1(\mathcal{H}_g;L)@>>> H_1(\mathcal{H}_g;H)@>>> H_1(\mathcal{H}_g;H/L)@>>> 0
\end{CD}
\]
is exact when $g\ge4$.
\end{remark}

\section{The Wajnryb's presentation of the handlebody mapping class group}
In this section,
we review the Wajnryb's presentation of the handlebody mapping class group $\mathcal{H}_g$ 
and compute the action of the handlebody mapping class group $\mathcal{H}_g$
to the first homology $H_1(\Sigma_g)$.
This is for preparing to calculate the twisted first homology $H_1(\mathcal{H}_g;H)$ when $g=2,3$ in Section~\ref{section:g23}.
\subsection{A presentation of the handlebody mapping class group}\label{subsection:wajnryb}
Let $g\ge2$.
We identify the surface in Figure~\ref{fig:scc1} with that in Figure~\ref{fig:surface}.
Let $\epsilon_i$ be a simple closed curve in Figure~\ref{fig:surface} for $i=1,\ldots,g-1$.
\begin{figure}[htbp]
\begin{center}
\includegraphics[width=130mm]{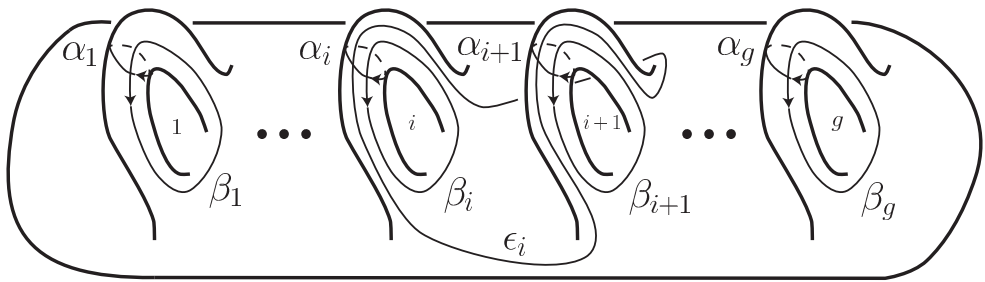}
\end{center}
\caption{the surface $\Sigma_g$}
\label{fig:surface}
\end{figure}
By cutting the surface $\Sigma_g$ along the simple closed curves $\alpha_1,\ldots,\alpha_g$,
we obtain a $(2g)$-holed sphere with boundary components
$\{\partial_{-i},\partial_i\}_{i=1}^g$ as in Figure~\ref{fig:sphere},
where $\alpha_i$ and $\beta_i$ correspond to the boundary components $\partial_{-i}\amalg \partial_i$ and the path from $\partial_{-i}$ to $\partial_i$, respectively.
\begin{figure}[htbp]
\begin{center}
\includegraphics[width=120mm]{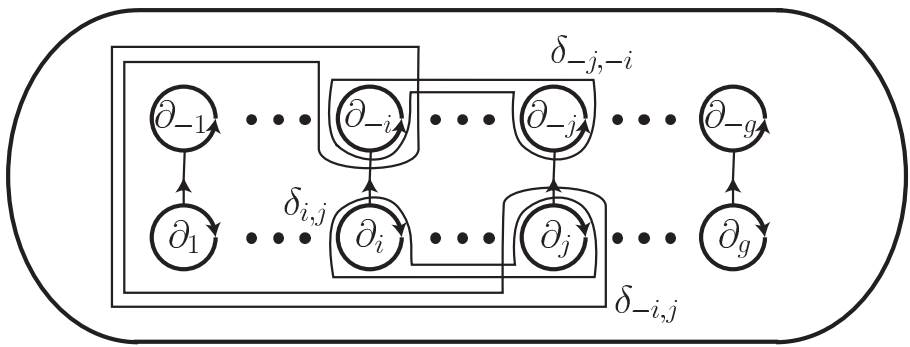}
\end{center}
\caption{the $(2g)$-holed sphere}
\label{fig:sphere}
\end{figure}
For integers $i,j$ satisfying $1\le i<j\le g$, 
we denote by $\delta_{-j,-i}$ and $\delta_{i,j}$ the simple closed curves in Figure~\ref{fig:sphere}.
For integers $i,j$ satisfying $1\le i\le g$ and $1\le j\le g$,
we also denote by $\delta_{-i,j}$ the simple closed curve in Figure~\ref{fig:sphere}.
For simplicity,
we denote by $a_i, b_i, e_i, d_{1,2}$ the Dehn twists along the curves $\alpha_i, \beta_i, \epsilon_j, \delta_{1,2}$, respectively.
Let us denote 
\begin{align*}
I_0&=\{-g,-(g-1),\ldots,-2, -1,1,2\ldots,g-1,g\},\\
s_{1}&=b_1a_1^2b_1,\\
t_i&=e_ia_ia_{i+1}e_i, \text{ for }i=1,\ldots, g-1.
\end{align*}
Since $t_i$ permutes the simple closed curves $\alpha_i$ and $\alpha_{i+1}$, and fixes other $\alpha_j$, 
we also have $t_i\in\mathcal{H}_g$.
In the following,
we denote $\varphi*\psi=\varphi\psi\varphi^{-1}$ 
for $\varphi, \psi\in\mathcal{H}_g$.
For $i,j\in I_0$ satisfying $i<j$,
we denote
\begin{align*}
d_{i,j}&=(t_{i-1}t_{i-2}\cdots t_1t_{j-1}t_{j-2}\cdots t_2)*d_{1,2}\text{ if }i>0,\\
d_{i,j}&=(t_{-i-1}^{-1}t_{-i-2}^{-1}\cdots t_1^{-1}s_{1}^{-1}t_{j-1}t_{j-2}\cdots t_2)*d_{1,2}\text{ if }i<0\text{ and }i+j>0,\\
d_{i,j}&=(t_{-i-1}^{-1}t_{-i-2}^{-1}\cdots t_1^{-1}s_1^{-1}t_jt_{j-1}\cdots t_2)*d_{1,2}\text{ if }j>0 \text{ and }i+j<0,\\
d_{i,j}&=(t_{-j-1}^{-1}t_{-j-2}^{-1}\cdots t_1^{-1}t_{-i-1}^{-1}t_{-i-2}^{-1}\cdots t_2^{-1}s_{1}^{-1}t_1^{-1}s_{1}^{-1})*d_{1,2}\text{ if }j<0,\\
d_{i,j}&=(t_{j-1}^{-1}d_{j-1,j}t_{j-2}^{-1}d_{j-2,j-1}\cdots t_1^{-1}d_{1,2})*(s_{1}^2a_{1}^4),\text{ if }i+j=0.
\end{align*}
Here, $d_{i,j}$ is actually the Dehn twist along $\delta_{i,j}$ in Figure~\ref{fig:sphere}
as explained in \cite[p. 211]{wajnryb98}.
However, to give a presentation of $\mathcal{H}_g$ with a small generating set,
we treat $d_{i,j}$ as the products above.
We also denote
\begin{align*}
d_{I}&=(d_{i_1,i_2}d_{i_1,i_3}\cdots d_{i_1,i_n}d_{i_2,i_3}\cdots d_{i_2,i_n}d_{i_3,i_4}\cdots d_{i_{n-1},i_n})(a_{i_1}\cdots a_{i_n})^{2-n},\\
&\text{where }I=\{i_1,\ldots,i_n\}\subset I_0\text{ and }i_1<\cdots<i_n,\\
c_{i,j}&=d_I,\text{ where }I=\{k\in I_0\,|\,i\le k\le j\} 
\text{ for }i\le j.
\end{align*}
Here, $d_{I}$ and $c_{i,j}$ are the Dehn twists along simple closed curves which enclose $\{\partial_{i_1},\ldots,\partial_{i_n}\}$ and $\{\partial_i,\ldots,\partial_j\}$, respectively.
See \cite[p. 211]{wajnryb98}, for details.
Let us denote
\[
\tilde{I}=\{(i,j)\in I_0^2\,|\,i=1, 1<j\}\cup\{(i,j)\in I_0^2\,|\,i<0, -i<j\le g+i\},
\]
and 
\begin{align*}
r_{i,j}&=b_ja_jc_{i,j}b_j, \text{ for }(i,j)\in \tilde{I},\\
k_j&=a_ja_{j+1}t_jd_{j,j+1}^{-1}\text{ for }j=1,\ldots,g-1,\\
s_j&=(k_{j-1}k_{j-2}\cdots k_1)*s_1\text{ for }j=2,\ldots, g,\\
z&=a_{1}a_{2}\cdots a_{g} 
s_{1}t_{1}t_{2}\cdots t_{g-1}s_{1}t_{1}\cdots t_{g-2}
s_{1}\cdots s_{1}t_{1}s_{1}d_{I}, \text{ where }I=\{1,\ldots, g\},\\
z_j&=k_{j-1}k_{j-2}\cdots k_{g+1-j}z \text{ for }j>\frac{g}{2}.
\end{align*}
Here, $r_{i,j}$ also lies in $\mathcal{H}_g$ as is explained in \cite[p. 211]{wajnryb98}.
For $\varphi, \psi\in \mathcal{H}_g$,
let us denote their commutator by $[\varphi,\psi]=\varphi\psi\varphi^{-1}\psi^{-1}$.
\begin{theorem}\cite[Theorem18]{wajnryb98}\label{thm:wajnryb}
The handlebody mapping class group of genus $g$ admits the following presentation:
The set of generators consists of $a_1, \ldots, a_g$, $d_{1,2}$, $s_{1}$, $t_1, \ldots, t_{g-1}$,
and $r_{i,j}$ for $(i,j)\in \tilde{I}$.
The set of defining relations is:
\begin{enumerate}
\renewcommand{\labelenumi}{(P\arabic{enumi})}
\item
$[a_i,a_j]=1$, $[a_i,d_{j,k}]=1$, for all $i,j,k\in I_0$,
\item 
Let $i,j,r,s\in I_0$.
\begin{enumerate}
\item $d_{r,s}^{-1}*d_{i,j}=d_{i,j}$ if $r<s<i<j$ or $i<r<s<j$,
\item $d_{r,i}^{-1}*d_{i,j}=d_{r,j}*d_{i,j}$ if $r<i<j$,
\item $d_{i,s}^{-1}*d_{i,j}=(d_{i,j}d_{s,j})*d_{i,j}$ if $i<s<j$,
\item $d_{r,s}^{-1}*d_{i,j}=[d_{r,j},d_{s,j}]*d_{i,j}$ if $r<i<s<j$,
\end{enumerate}
\item $d_{I_0}=1$,
\item $d_{I_k}=a_{|k|}$ where $I_k=I_0-\{k\}$ for $k\in I_0$,
\item $t_it_{i+1}t_i=t_{i+1}t_it_{i+1}$ for $i=1,\ldots,g-2$,
and $[t_i,t_j]=1$ if $1\le i<j-1<g-1$,
\item $t_i^2=d_{i,i+1}d_{-i-1,-i}a_i^{-2}a_{i+1}^{-2}$ for $i=1,\ldots,g-1$,
\item $[s_{1}, a_i]=1$ for $i=1,\ldots,g$,
$t_i*a_i=a_{i+1}$ for $i=1,\ldots,g-1$,
$[a_i,t_j]=1$ for $i,j\in I_0$ satisfying $j\ne i, i-1$,
and $[t_{i}, s_{1}]=1$ for $i=2,\ldots, g-1$,
\item $[s_{1}, d_{2, 3}]=1$,
$[s_{1}, d_{-2,2}]=1$,
$s_{1}t_{1}s_{1}t_{1}=t_{1}s_{1}t_{1}s_{1}$,
and $[t_i,d_{1,2}]=1$ for $i=1,3,\ldots,g-1$,
\item $r_{i,j}^2=s_jc_{i,j}s_jc_{i,j}^{-1}$ for $(i,j)\in\tilde{I}$,
\item Let $(i,j)\in \tilde{I}$. 
\begin{enumerate}
\item $r_{i,j}*a_j=c_{i,j}$ and $[r_{i,j}, a_k]=1$ if $k\ne j$,
\item $[r_{i,j},t_k]=1$ if $k\ne |i|,j$ or $k=i=1<j-1$,
\item $[r_{i,j},s_k]=1$ if $k<|i|$, $j<k$ or $k=-i$,
\item $[r_{i,j},d_{k,m}]=1$ if $k,m\in\{i,\ldots, j-1\}$ or $k,m\notin \{-j,i,i+1,\ldots,j\}$,
\item $[r_{i,j},z_j]=1$ if $(i,j)=(1,g)$ or $j=g+i$,
\item $r_{i,j}*d_{i,j}=d_J$ where $J=\{k\in I_0; i<k\le j\}$,
\item $r_{1,j}*d_{-j,1-j}=(t_{j-2}t_{j-3}\cdots t_1)*c_{-1,j}$,
\item $r_{i,j}*d_{-j,1-j}=(t_{j-2}t_{j-3}\cdots t_{1-i})*c_{i-1,j}$ if $i<0$ and $j+i>1$,
\item $r_{i,j}^{-1}*d_{-j-1,-j}=s_{j+1}^{-1}*c_{i,j+1}$ if $j<g$,
\end{enumerate}
\item $r_{i,j}*t_{j-1}=t_{j-1}^{-1}*r_{i,j}$ if $(i,j)\in \tilde{I}$ and $-i+1\ne j$,
\item 
\begin{enumerate}
\item Let $h_2=k_{j-1}^{-1}t_{j-2}^{-1}t_{j-3}^{-1}\cdots t_1^{-1}k_{j-1}k_{j-2}\cdots k_2$.
\[
r_{1, j}=s_jc_{1,j}s_jc_{1,j}^{-1}
k_{j-1}a_jc_{1,j-2}t_{j-1}c_{1,j-1}^{-1}t_{j-1}^{-1}r_{1,j-1}^{-1}s_{j-1}
h_2r_{1,2}^{-1}h_2^{-1}k_{j-1}^{-1}
\]
for $3\le j\le g$.
\item Let $h_3=s_1k_{j-1}k_{j-2}\cdots k_2$.
\[
r_{-1, j}=h_{3}r_{1,2}^{-1}h_{3}^{-1}
s_{j}r_{1, j}^{-1}c_{-1, j-1}^{-1}c_{1, j-1}a_{1}s_{j}c_{-1, j}s_{j}c_{-1, j}^{-1}
\]
for $2\le j\le g-1$.
\item Let $h_3=s_{-i}t_{-1-i}^{-1}t_{-2-i}^{-1}\cdots t_1^{-1}k_{j-1}k_{j-2}\cdots k_3k_2$.
\[
r_{i,j}=h_3r_{1,2}^{-1}h_3^{-1}s_jr_{i+1,j}^{-1}c_{i,j-1}^{-1}c_{i+1,j-1}a_{-i}s_jc_{i,j}s_jc_{i,j}^{-1}
\]
for $i<-1$ and $(i,j)\in \tilde{I}$.
\end{enumerate}
\end{enumerate}
\end{theorem}

\begin{remark}
Note that there are some mistakes in the Wajnryb's presentation in \cite{wajnryb98}.
The mapping class $z_j$ is defined as
the conjugation of $z$ by $k_{j-1}k_{j-2}\cdots k_{g+1-j}$ in \cite{wajnryb98}.
However, as mentioned in \cite{popescu11}, 
it should be defined as the product $k_{j-1}k_{j-2}\cdots k_{g+1-j}z$.
In (P11), the condition $-i+1\ne j$ is needed. 
The relations of type (P11) are obtained in the situation 
when the pair of simple closed curves $\partial _{k}$ and $\partial _{-k}$ 
are separated by $\gamma _{i, j}$ for $k=j, j-1$ (see CASE~1 in \cite[p.223]{wajnryb98}),
and the equation $r_{-(j-1), j}*t_{j-1}=t_{j-1}^{-1}*r_{-(j-1), j}$ 
in fact does not hold for any $2\leq j\leq g$. 
We also erase the relation $s_{1}^{2}=d_{-1,1}a_{1}^{-4}$ in (P6) written in \cite{wajnryb98}.
This is because we already defined $d_{-1,1}$ as $s_{1}^{2}a_{1}^{4}$.
\end{remark}

\subsection{Action on the first homology $H_1(\Sigma_g)$}
Here, we compute the action of the handlebody mapping class group $\mathcal{H}_g$
on the first homology $H_1(\Sigma_g)$ of the boundary surface.
Recall that $x_1,\ldots, x_g, y_1,\ldots, y_g$ are the homology classes represented by the simple closed curves
$\alpha_1,\ldots, \alpha_g, \beta_1,\ldots, \beta_g$ in Figures~\ref{fig:scc1} and \ref{fig:surface}.
\begin{lemma}\label{lem:action}
For $1\leq i\leq g$, 
\[ a_{i}(x_{l})=x_{l}, 
\quad
a_{i}(y_{l})=\left\{ 
\begin{array}{ll}
x_{i}+y_{i} & \text{if }l=i, \\
y_{l} & \text{otherwise}, \\
\end{array}\right. \] 
and
\[
s_{i}(x_{l})=\left\{ 
\begin{array}{ll}
-x_{i} & \text{if }l=i, \\
x_{l} & \text{otherwise}, \\
\end{array}\right.
\quad
s_{i}(y_{l})=\left\{ 
\begin{array}{ll}
2x_{i}-y_{i} & \text{if }l=i, \\
y_{l} & \text{otherwise}. \\
\end{array}\right.\] 

For each $i, j\in I_{0}$ such that $i<j$,  
\[
d_{i, j}(x_{l})=x_{l}, 
\quad
d_{i, j}(y_{l})=\left\{ 
\begin{array}{ll}
\varepsilon(i)x_{|i|}+\varepsilon(j)x_{|j|}+y_{l} & \text{if }l=|i|, |j|, \\
y_{l} & \text{otherwise}, \\
\end{array}\right.
\] 
where $\varepsilon (i)=1$ if $i>0$,
and $\varepsilon (i)=-1$ if $i<0$.  

For $1\leq i\leq g-1$, 
\[
t_{i}(x_{l})=\left\{ 
\begin{array}{ll}
x_{i+1} & \text{if }l=i, \\
x_{i} & \text{if }l=i+1, \\
x_{l} & \text{otherwise}, \\
\end{array}\right. 
\quad
t_{i}(y_{l})=\left\{ 
\begin{array}{ll}
x_{i}+y_{i+1} & \text{if }l=i, \\
x_{i+1}+y_{i} & \text{if }l=i+1, \\
y_{l} & \text{otherwise}, \\
\end{array}\right. \]
and
\[
k_{i}(x_{l})=\left\{ 
\begin{array}{ll}
x_{i+1} & \text{if }l=i, \\
x_{i} & \text{if }l=i+1, \\
x_{l} & \text{otherwise}, \\
\end{array}\right.
\quad
k_{i}(y_{l})=\left\{ 
\begin{array}{ll}
y_{i+1} & \text{if }l=i, \\
y_{i} & \text{if }l=i+1, \\
y_{l} & \text{otherwise}. \\
\end{array}\right. \] 

For $1<j\leq g$, 
\begin{align*}
r_{1, j}(x_{l})&=\left\{ 
\begin{array}{ll}
-x_{1}-\cdots -x_{j} & \text{if }l=j, \\
x_{l} & \text{otherwise}, \\
\end{array}\right.\\
r_{1, j}(y_{l})&=\left\{ 
\begin{array}{ll}
x_{1}+\cdots +x_{j}+y_{l}-y_{j} & \text{if }1\leq l\leq j-1, \\
x_{1}+\cdots +x_{j-1}+2x_{j}-y_{j} & \text{if }l=j, \\
y_{l} & \text{otherwise}, \\
\end{array}\right. 
\end{align*}
and for $(i,j)\in \tilde{I}$ such that $i<0$, 
\begin{align*}
r_{i, j}(x_{l})&=\left\{ 
\begin{array}{ll}
-x_{-i+1}-\cdots -x_{j} & \text{if }l=j, \\
x_{l} & \text{otherwise}, \\
\end{array}\right.\\
r_{i, j}(y_{l})&=\left\{ 
\begin{array}{ll}
x_{-i+1}+\cdots +x_{j}+y_{l}-y_{j} & \text{if }-i+1\leq l\leq j-1, \\
x_{-i+1}+\cdots +x_{j-1}+2x_{j}-y_{j} & \text{if }l=j, \\
y_{l} & \text{otherwise}. \\
\end{array}\right.
\end{align*}
\end{lemma}
\begin{proof}
The equations for the mapping classes $a_{i}$ and $d_{i, j}$
are obvious 
because $a_{i}$ and $d_{i, j}$ are Dehn twists along 
$\alpha _{i}$ and $\delta_{i, j}$ respectively. 
Similarly we have 
\[
b_{i}(x_{l})=\left\{ 
\begin{array}{ll}
x_{i}-y_{i} & \text{if }l=i, \\
y_{l} & \text{otherwise}, \\
\end{array}\right.
\quad
b_{i}(y_{l})=y_{l},
\] 
for $1\leq i\leq g$ and 
\[ 
e_{i}(x_{l})=\left\{ 
\begin{array}{ll}
x_{i}-y_{i}+y_{i+1} & \text{if }l=i,\\
x_{i+1}+y_{i}-y_{i+1} & \text{if }l=i+1,\\
x_{l} & \text{otherwise},\\
\end{array}\right.
\quad
e_{i}(y_{l})=y_{l},
\] 
for $1\leq i\leq g-1$. 

Since $t_{i}=e_{i}a_{i}a_{i+1}e_{i}$, 
\begin{align*}
t_{i}(x_{i})&=(e_{i}a_{i}a_{i+1})(x_{i}-y_{i}+y_{i+1})
=e_{i}(x_{i+1}-y_{i}+y_{i+1})=x_{i+1}, \\
t_{i}(x_{i+1})&=(e_{i}a_{i}a_{i+1})(x_{i+1}+y_{i}-y_{i+1})
=e_{i}(x_{i}+y_{i}-y_{i+1})=x_{i}, \\
t_{i}(y_{i})&=(e_{i}a_{i}a_{i+1})(y_{i})
=e_{i}(x_{i}+y_{i})=x_{i}+y_{i+1}, \\
t_{i}(y_{i+1})&=(e_{i}a_{i}a_{i+1})(y_{i+1})
=e_{i}(x_{i+1}+y_{i+1})=x_{i+1}+y_{i}
\end{align*}
and $t_{i}$ acts trivially on other $x_{l}$'s and $y_{l}$'s. 

Since $k_{i}=a_{i}a_{i+1}t_{i}d_{i, i+1}^{-1}$, 
\begin{align*}
k_{i}(x_{i})&=(a_{i}a_{i+1}t_{i})(x_{i})
=(a_{i}a_{i+1})(x_{i+1})=x_{i+1}, \\
k_{i}(x_{i+1})&=(a_{i}a_{i+1}t_{i})(x_{i+1})
=(a_{i}a_{i+1})(x_{i})=x_{i}, \\
k_{i}(y_{i})&=(a_{i}a_{i+1}t_{i})(-x_{i}-x_{i+1}+y_{i})
=(a_{i}a_{i+1})(-x_{i+1}+y_{i+1})=y_{i+1}, \\
k_{i}(y_{i+1})&=(a_{i}a_{i+1}t_{i})(-x_{i}-x_{i+1}+y_{i+1})
=(a_{i}a_{i+1})(-x_{i}+y_{i})=y_{i},
\end{align*}
and $k_{i}$ acts trivially on other $x_{l}$'s and $y_{l}$'s. 

Since $s_{1}=b_{1}a_{1}^{2}b_{1}$, 
\begin{align*}
s_{1}(x_{1})&=(b_{1}a_{1}^{2})(x_{1}-y_{1})
=b_{1}(-x_{1}-y_{1})=-x_{1}, \\
s_{1}(y_{1})&=(b_{1}a_{1}^{2})(y_{1})
=b_{1}(2x_{1}+y_{1})=2x_{1}-y_{1},
\end{align*}
and $s_{1}$ acts trivially on other $x_{l}$'s and $y_{l}$'s. 
The elements $s_{i}$'s are inductively defined 
by the recurrence relation $s_{i+1}=k_{i}s_{i}k_{i}^{-1}$. 
The element $k_{i}$ replaces $x_{i}$ and $x_{i+1}$ with each other 
and $y_{i}$ and $y_{i+1}$ also. 
Hence the equation for $s_{i}$ follows by induction. 

Lastly, we verify the equations for $r_{i, j}$. 
In the case $0<i<j$, 
\[ c_{i, j}(x_{l})=x_{l}, 
\quad 
c_{i, j}(y_{l})=\left\{
\begin{array}{ll}
x_{i}+\cdots +x_{j}+y_{l} & \text{if }i\leq l\leq j,\\
y_{l} & \text{otherwise}.
\end{array}\right.\]
Since $r_{i, j}=b_{j}a_{j}c_{i, j}b_{j}$, 	
\begin{align*}
r_{1, j}(x_{j})&=(b_{j}a_{j}c_{1, j})(x_{j}-y_{j}) \\
&=(b_{j}a_{j})(-x_{1}-\cdots -x_{j-1}-y_{j}) \\
&=-x_{1}-\cdots -x_{j} , 
\end{align*}
and for $1\leq l\leq j$
\begin{align*}
r_{1, j}(y_{l})&=(b_{j}a_{j}c_{1, j})(y_{l}) \\
&=(b_{j}a_{j})(x_{1}+\cdots +x_{j}+y_{l}) \\
&=\left\{
\begin{array}{ll}
x_{1}+\cdots +x_{j}+y_{l}-y_{j} & \text{if }1\leq l\leq j-1, \\
x_{1}+\cdots +x_{j-1}+2x_{j}-y_{j} & \text{if }l=j.
\end{array}
\right.
\end{align*}
The element $r_{1, j}$ acts trivially on other $x_{l}$'s and $y_{l}$'s. 

In the case  $(i,j)\in \tilde{I}$ and $i<0$, 
\[ c_{i, j}(x_{l})=x_{l}, 
\quad 
c_{i, j}(y_{l})=\left\{
\begin{array}{ll}
x_{-i+1}+\cdots +x_{j}+y_{l} & \text{if }-i+1\leq l\leq j, \\
y_{l} & \text{otherwise}.
\end{array}\right. .\]
Hence 
\begin{align*}
r_{i, j}(x_{j})&=(b_{j}a_{j}c_{i, j})(x_{j}-y_{j}) \\
&=(b_{j}a_{j})(-x_{-i+1}-\cdots -x_{j-1}-y_{j}) \\
&=-x_{-i+1}-\cdots -x_{j}, 
\end{align*}
and for $-i+1\leq l\leq j$
\begin{align*}
r_{i, j}(y_{l})&=(b_{j}a_{j}c_{i, j})(y_{l}) \\
&=(b_{j}a_{j})(x_{-i+1}+\cdots +x_{j}+y_{l}) \\
&=\left\{
\begin{array}{ll}
x_{-i+1}+\cdots +x_{j}+y_{l}-y_{j} & \text{if }-i+1\leq l\leq j-1, \\
x_{-i+1}+\cdots +x_{j-1}+2x_{j}-y_{j} & \text{if }l=j.
\end{array}
\right.
\end{align*}
The element $r_{i, j}$ acts trivially on other $x_{l}$'s and $y_{l}$'s. 
\end{proof}
%

\section{Proof of Theorem~\ref{thm:main theorem} for $g=2,3$}\label{section:g23}
In this section,
we prove Theorem~\ref{thm:main theorem} for $g=2, 3$.
We denote by $A$ the ring $\mathbb{Z}$ or $\mathbb{Z}/n\mathbb{Z}$ for an integer $n\ge2$, 
and $H_A=H_1(\Sigma_g;A)$.
Recall that, for a group $G$ and a left $G$-module $M$,
a map $d:G\to M$ is called a crossed homomorphism
if it satisfies $d(hh')=d(h)+hd(h')$ for $h,h'\in G$.
For a group $G$ and a left $G$-module $M$,
we consider group cohomology $H^*(G;M)$
as that of the standard chain complex induced by the bar resolution.
Then, the space of 1-cocycles $Z^1(G;M)$ is identified with
the space of crossed homomorphisms from $G$ to $M$,
and the space of 1-coboundaries $B^1(G;M)$ is identified with
the image of the coboundary map $\delta: M\to Z^1(G;M)$
defined by $\delta(m)(h)=hm-m$ for $m\in M$.
See \cite[Section~2.3]{evens91} for details.

As written in \cite[Theorem~19]{wajnryb98},
the handlebody mapping class group $\mathcal{H}_g$ is generated by
$a_1$, $s_1$, $r_{1,2}$, $t_1$, and $u=t_1t_2\cdots t_{g-1}$.
Therefore, crossed homomorphisms $d: \mathcal{H}_g\to H_A$
are uniquely determined by the values $d(a_1)$, $d(s_1)$, $d(r_{1,2})$, $d(t_1)$, and $d(u)$.
Moreover,
a 5-tuple of elements in $H_A$ becomes values of $a_1$, $s_1$, $r_{1,2}$, $t_1$, and $u$
under some crossed homomorphism $d$ on $\mathcal{H}_g$
if and only if they are compatible with the relations (P1)-(P12) in Theorem~\ref{thm:wajnryb}.
The basis $\{x_1,\ldots, x_g, y_1,\ldots, y_g\}$ of $H_A$
induces an isomorphism $H_A\cong A^{2g}$.
For $v\in H_A$,
we denote its projection to the $i$-th coordinate of $A^{2g}$
by $v_i\in A$ for $i=1,2,\ldots, 2g$.
\begin{lemma}\label{lem:homology-cond} 
\begin{align*}
H^{1}(\mathcal{H}_{2} ;H_A)&\cong
\{d\in Z^{1}(\hb{g} ; H_A); 
d(r_{1, 2})_{1}=d(s_{1})_{3}-d(r_{1, 2})_{4}=d(u)_{2}=d(u)_{4}=0 \}, \\
H^{1}(\mathcal{H}_{3} ; H_A)&\cong
\{d\in Z^{1}(\hb{g} ; H_A); \\
& \qquad\quad 
d(r_{1, 2})_{1}=d(s_{1})_{4}-d(r_{1, 2})_{5}=d(u)_{2}=d(u)_{3}=d(u)_{5}=d(u)_{6}=0 \} .
\end{align*}
\end{lemma}

\begin{proof}
Let 
$f_{2}\colon Z^1(\mathcal{H}_{2}; H_A)\to A^{4}$ 
and $f_{3}\colon Z^1(\mathcal{H}_{3}; H_A)\to A^{6}$ 
be homomorphisms defined by
\begin{align*}
f_{2}(d)&=(d(r_{1, 2})_{1}, d(s_{1})_{3}-d(r_{1, 2})_{4}, d(u)_{2}, d(u)_{4}), \text{ and } \\
f_{3}(d)&=(d(r_{1,2})_{1}, d(s_{1})_{4}-d(r_{1, 2})_{5}, d(u)_{2}, d(u)_{3}, d(u)_{5}, d(u)_{6}),
\end{align*}
respectively. 
Then, the composition maps $f_g\circ \delta: H_A\to A^{2g}$ are written as
\begin{align*}
f_{2}\circ \delta(v)
&=(-v_{2}+v_{3}+v_{4}, -v_{3}+2v_{4}, v_{1}-v_{2}+v_{4}, v_{3}-v_{4}), \\ 
f_{3}\circ \delta(v)
&=(-v_{2}+v_{4}+v_{5}, -v_{4}+2v_{5}, v_{1}-v_{2}+v_{5}+v_{6}, v_{2}-v_{3}+v_{6}, v_{4}-v_{5}, v_{5}-v_{6}),
\end{align*}
for $v\in H_A$.
Since these maps are isomorphisms,
we have
\[
H^1(\mathcal{H}_g; H_A)
=Z^1(\mathcal{H}_g; H_A)/B^1(\mathcal{H}_g; H_A)\cong \Ker f_g
\]
for $g=2,3$.
\end{proof}

\begin{lemma}\label{u-conj}
Suppose $d\in Z^{1}(\hb{g}; H_A)$ 
satisfies $d(u)_{2}=\cdots=d(u)_{g}=d(u)_{g+2}=\cdots=d(u)_{2g}=0$ as in Lemma~\ref{lem:homology-cond}.
Then,
\begin{enumerate}
\item $d(a_{i})=u^{i-1}d(a_{1})$. 
\item $d(t_{i})=u^{i-1}d(t_{1})$. 
\end{enumerate}
\end{lemma}
\begin{proof}
Note that $a_{i}=u^{i-1}a_{1}u^{-(i-1)}$. 
It can be checked using the relation (P7). 
Hence  we have 
\[ d(a_{i+1})=d(u)+ud(a_{i})-ua_{i}u^{-1}d(u)
=d(u)+ud(a_{i})-a_{i+1}d(u). \]
Since $(a_{i+1}v)_{1}=v_{1}$ and $(a_{i+1}v)_{g+1}=v_{g+1}$ 
for any $v\in H_A$, 
we have $a_{i+1}d(u)=d(u)$,
and thus $d(a_{i+1})=ud(a_{i})$. 
By induction on $i$, 
we have the equation (1). 
The equation (2) can be similarly verified. 
\end{proof}

\begin{lemma}\label{vanish}
Suppose $d\in Z^{1}(\hb{g}; H_A)$ 
satisfies $d(u)_{2}= \dots =d(u)_{g}=d(u)_{g+2}= \dots =d(u)_{2g}=0$ as in Lemma~\ref{lem:homology-cond}.
Then
\begin{enumerate}
\item $d(a_{1})_{g+1}= \dots =d(a_{1})_{2g}=0$. 
\item $2d(a_{1})_{2}= \dots =2d(a_{1})_{g}=0$. 
\item $d(s_{1})_{g+2}= \dots =d(s_{1})_{2g}=0$. 
\item $d(s_{1})_{g+1}=-2d(a_{1})_{1}$. 
\item $d(a_{1})_{2}+d(r_{1, 2})_{g+1}=0$.  
\end{enumerate}
\end{lemma}
\begin{proof}
For any $i$ and $j$, 
\[ d(a_{i}a_{j})=d(a_{i})+a_{i}d(a_{j})=d(a_{i})+d(a_{j})+d(a_{j})_{g+i}x_{i}. \]
Since $a_{1}$ and $a_{i}$ commute for any $1\leq i\leq g$ 
by the relation (P1), 
it must be $d(a_{1}a_{i})=d(a_{i}a_{1})$,
and thus $d(a_{1})_{g+i}=0$ for any $2\leq i\leq g$. 

Since $a_{1}$ and $r_{1, 2}$ commute by the relation (P10)(a),
it must be 
\[
(1-a_{1})d(r_{1, 2})=(1-r_{1, 2})d(a_{1}).
\] 
Since $((1-r_{1, 2})v)_{g+2}=v_{g+1}+2v_{g+2}$ for any $v\in H_A$ 
while $((1-a_{1})v)_{g+2}=0$, 
we have $d(a_{1})_{g+1}=0$ and thus the equation (1). 
Since $((1-r_{1, 2})v)_{1}=v_{2}-v_{g+1}-v_{g+2}$ for any $v\in H_A$ 
while $((1-a_{1})v)_{1}=-v_{g+1}$, 
we have the equation (5). 

Note that $a_{i}$ and $s_{j}$ commute for any $1\leq i, j\leq g$. 
It can be verified using the relations (P1) and (P7). 
Hence it must be 
\[
(1-s_{j})d(a_{i})=(1-a_{i})d(s_{j}).
\]
Suppose $i=1$ and $2\leq j\leq g$. 
Then we have the equation (2) 
because  $((1-s_{j})v)_{j}=2v_{j}-2v_{g+j}$ 
and $((1-a_{1})v)_{j}=0$ for any $v\in H_{A}$. 
Suppose $2\leq i\leq g$ and $j=1$. 
Then we have the equation (3) 
because $((1-s_{1})v)_{i}=0$
and $((1-a_{i})v)_{i}=-v_{g+i}$ for any $v\in H_{A}$. 
Suppose $i=j=1$. 
Then we have the equation (4) 
because $((1-s_{1})v)_{1}=2v_{1}-2v_{g+1}$  
and $((1-a_{1})v)_{1}=-v_{g+1}$ for any $v\in H_A$. 
\end{proof}
\subsection{$H^{1}(\hb{2}; H_A)$}\label{subsection:genus2}
Here, we assume $g=2$ 
and prove that $H^{1}(\hb{2}; H_A)\cong \Hom((\mathbb{Z}/2\mathbb{Z})^2,A)$.
Then, the universal coefficient theorem implies $H_1(\hb{2}; H)\cong (\mathbb{Z}/2\mathbb{Z})^2$ 
and we complete the proof of Theorem~\ref{thm:main theorem} when $g=2$.

Let $d\in Z^{1}(\hb{2}; H_A)$ be a crossed homomorphism satisfying the condition 
$d(r_{1, 2})_{1}=d(s_{1})_{3}-d(r_{1, 2})_{4}=d(u)_{2}=d(u)_{4}=0$
as in Lemma~\ref{lem:homology-cond}. 
Note that in this case $u=t_{1}$. 
By Lemma \ref{vanish}, we can set 
\begin{align*}
d(a_{1})&=w_{1, 1}x_{1}+w_{1, 2}x_{2}, \\
d(s_{1})&=w_{2, 1}x_{1}+w_{2, 2}x_{2}+w_{2, 3}y_{1}, \\
d(t_{1})&=w_{3, 1}x_{1}+w_{3, 3}y_{1}, \\
d(r_{1, 2})&=w_{4, 2}x_{2}+w_{4, 3}y_{1}+w_{4, 4}y_{2}. 
\end{align*}
By the condition on $d$ and Lemma~\ref{vanish}, we also have
\begin{equation}\label{relation2-0}
2w_{1, 2}=0, w_{2, 3}=w_{4, 4}=-2w_{1, 1},\text{ and }w_{1, 2}+w_{4, 3}=0.
\end{equation} 
\begin{lemma}\label{relation2-1}
\[ w_{1, 2}=w_{4, 3}=0. \]
Moreover, we have
\[
d(d_{1, 2})=w_{1, 1}(x_{1}+x_{2}).
\]
\end{lemma}

\begin{proof}
Since $w_{1, 2}+w_{4, 3}=0$, 
it suffices to prove that $w_{4, 3}=0$. 
Note that by Lemma \ref{u-conj}, we have
$d(a_{2})=t_{1}d(a_{1})=w_{1, 2}x_{1}+w_{1, 1}x_{2}$. 
Since $d_{1, 2}=r_{1, 2}a_{2}r_{1, 2}^{-1}$ by the relation (P10)(a), 
\begin{align*}
d(d_{1, 2})&=d(r_{1, 2})+r_{1, 2}d(a_{2})-r_{1, 2}a_{2}r_{1, 2}^{-1}d(r_{1, 2}) \\
&=(1-d_{1, 2})d(r_{1, 2})+r_{1, 2}d(a_{2}) \\
&=(-w_{1, 1}-w_{4, 3}-w_{4, 4})(x_{1}+x_{2})
+w_{1, 2}x_{1}. 
\end{align*}

Since $a_{2}=r_{1, 2}d_{1, 2}r_{1, 2}^{-1}$ by the relation (P10)(f), 
\[
d(a_{2})=(1-a_{2})d(r_{1, 2})+r_{1, 2}d(d_{1, 2})
=w_{1, 2}x_{1}+(w_{1, 1}+w_{4, 3})x_{2}. 
\]
Thus we obtain $w_{4, 3}=0$.
By the equation $w_{4, 4}=-2w_{1, 1}$ in (\ref{relation2-0}),
we have $d(d_{1,2})=w_{1, 1}(x_{1}+x_{2})$.
\end{proof}
\begin{lemma}\label{relation2-2}
\[ w_{2, 3}=w_{3, 1}=w_{3, 3}=w_{4, 4}=0. \]
In particular, we have $d(t_{1})=0$ and $2w_{1, 1}=0$. 
\end{lemma}

\begin{proof}
Recall that $d_{-2, -1}=(s_{1}t_{1}s_{1})^{-1}d_{1, 2}(s_{1}t_{1}s_{1})$. 
In the case $g=2$, the Dehn twist $d_{-2, -1}$ coincides with $d_{1, 2}$. 
Hence the elements $d_{1, 2}$ and $s_{1}t_{1}s_{1}$ commute  
and it must be 
$(1-d_{1, 2})d(s_{1}t_{1}s_{1})=(1-s_{1}t_{1}s_{1})d(d_{1, 2})$. 
Since 
\begin{align*}
((1-d_{1, 2})d(s_{1}t_{1}s_{1}))_{1} 
&=((1-d_{1, 2})d(s_{1}t_{1}s_{1}))_{2}
=-2w_{2, 3}+w_{3, 3},\text{ while }\\
((1-s_{1}t_{1}s_{1})d(d_{1, 2}))_{1}
&=((1-s_{1}t_{1}s_{1})d(d_{1, 2}))_{2}
=2w_{1, 1},
\end{align*}
we have $2w_{1, 1}=-2w_{2, 3}+w_{3, 3}$.
The equation $w_{2,3}=-2w_{1,1}$ in (\ref{relation2-0}) shows $w_{2, 3}=w_{3, 3}$.

Since $w_{2, 3}=w_{3, 3}=w_{4, 4}$, 
it remains to prove that $w_{3, 1}=w_{3, 3}=0$. 
Note that each of $d(a_{1})$, $d(a_{2})$, and $d(d_{1, 2})$ 
are in $L_A=\Ker(H_1(\Sigma_2;A)\to H_1(H_2;A))$. 
By the relation (P6), 
\[ t_{1}^2=d_{1, 2}^{2}a_{1}^{-2}a_{2}^{-2}. \]
Since each of $a_{1}$, $a_{2}$ and $d_{1, 2}$ acts on $L$ trivially, we have
\[ d(d_{1, 2}^{2}a_{1}^{-2}a_{2}^{-2})
=2(d(d_{1, 2})-d(a_{1})-d(a_{2}))=0. \]
On the other hand, we have
\[ d(t_{1}^{2})=d(t_{1})+t_{1}d(t_{1})
=w_{3, 1}(x_{1}+x_{2})+w_{3, 3}(x_{3}+x_{4})+w_{3, 3}x_{1}. \]
These equations show $w_{3, 1}=w_{3, 3}=0$.
\end{proof}
\begin{lemma}\label{relation2-3}
\[ w_{4, 2}=0. \]
In particular, we have $d(r_{1, 2})=0$. 
\end{lemma}

\begin{proof}
The relation $t_{1}r_{1, 2}t_{1}=r_{1, 2}t_{1}r_{1, 2}$ in (P11) shows
\[
(1-t_{1}+r_{1, 2}t_{1})d(r_{1, 2})=(1-r_{1, 2}+t_{1}r_{1, 2})d(t_{1}).
\]
By Lemma~\ref{relation2-2}, the right hand side is equal to zero.
Since 
\[ (1-t_{1}+r_{1, 2}t_{1})d(r_{1, 2}) 
=w_{4, 2}x_{2}, \]
we have $w_{4, 2}=0$.
\end{proof}
\begin{lemma}\label{relation2-4}
\[ 2w_{2, 2}=0. \]
\end{lemma}
\begin{proof}
The relation $r_{1, 2}^{2}=s_{2}d_{1, 2}s_{2}d_{1, 2}^{-1}$ in (P9) 
and Lemma \ref{relation2-3} show
\[
d(s_{2}d_{1, 2}s_{2}d_{1, 2}^{-1})=d(r_{1, 2}^{2})=0.
\] 
Recall that $k_{1}=a_{1}a_{2}t_{1}d_{1, 2}^{-1}$ and $s_{2}=k_{1}s_{1}k_{1}^{-1}$. 
Hence we have
\begin{align*}
d(k_{1})&=d(a_{1})+d(a_{2})-t_{1}d(d_{1, 2})=0, \text{ and }\\
d(s_{2})&=d(k_{1})+k_{1}d(s_{1})-s_{2}d(k_{1})=k_{1}d(s_{1}).
\end{align*}
Therefore, we have 
\[
d(s_{2}d_{1, 2}s_{2}d_{1, 2}^{-1})
=(1+s_{2}d_{1, 2})d(s_{2})+(s_{2}-1)d(d_{1, 2})
=2w_{2, 2}x_{1} 
\]
and thus $2w_{2, 2}=0$.
\end{proof}
\begin{lemma}\label{relation2-5}
\[ w_{2, 1}=w_{2, 2}. \]
\end{lemma}
\begin{proof}
Recall that $z=a_{1}a_{2}s_{1}t_{1}s_{1}d_{1, 2}$ and $z_{2}=k_{1}z$. 
Hence we have
\begin{align*}
d(z)&=d(a_{1})+d(a_{2})+(1+s_{1}t_{1})d(s_{1})+s_{1}t_{1}s_{1}d(d_{1, 2}) \\
&=(w_{2, 1}+w_{2, 2})(x_{1}+x_{2}),\text{ and }\\
d(z_{2})&=d(k_{1})+k_{1}d(z)=d(z).
\end{align*}
Hence
\[ (1-r_{1, 2})d(z_{2})=(w_{2, 1}+w_{2, 2})(x_{1}+2x_{2}) \]
Since $r_{1, 2}$ and $z_{2}$ commute by the relation (P10)(e),
it must be $(1-r_{1, 2})d(z_{2})=(1-z_{2})d(r_{1, 2})=0$. 
Thus we have $w_{2, 1}=w_{2, 2}$.
\end{proof}
Summarizing Lemmas \ref{relation2-1}, 
\ref{relation2-2}, \ref{relation2-3}, \ref{relation2-4} and \ref{relation2-5},
we have 
\[ d(a_{1})=w_{1, 1}x_{1}, d(s_{1})=w_{2, 1}(x_{1}+x_{2}) 
\text{ and } d(t_{1})=d(r_{1, 2})=0, \] 
where $2w_{1, 1}=2w_{2, 1}=0$. 
It can be verified that such $d$ 
is compatible with the relations (P1)--(P12). 
Now we have 
\begin{align*}
H^{1}(\hb{2}; H_A)\cong 
\Ker f_{2}\cong
\{(w_{1, 1}, w_{2, 1})\in A^{2}; 2w_{1, 1}=2w_{2, 1}=0 \} .
\end{align*}

\begin{prop}\label{prop:H/L genus2}
\[
H_{1}(\hb{2}; L)\cong \mathbb{Z}/2\mathbb{Z},
H_1(\hb{2};H/L)\cong H_1(\hb{2};H),
\]
and the homomorphism $H_2(\hb{2};H/L)\to H_1(\hb{2};L)$ induced by
the exact sequence $0\to L\to H\to H/L\to 0$ is surjective.
\end{prop}
\begin{proof}
As well as Lemma \ref{vanish},
we can verify that 
\[ H^{1}(\hb{2}; H_A/L_A)\cong
\{d'\in Z^{1}(\hb{2}; H_A/L_A); d'(s_{1})_{1}-d'(r_{1, 2})_{2}=d'(u)_{2}=0\}. 
\]
Similar calculations in Section~\ref{subsection:genus2} show 
that a crossed homomorphism $d'\colon \hb{2}\to H_A/L_A$ 
such that $d'(s_{1})_{1}-d'(r_{1, 2})_{2}=d'(u)_{2}=0$ 
is compatible with the relations (P1)--(P12) 
if and only if 
\[ d'(a_{1})=d'(t_{1})=0, d'(s_{1})=w'_{2, 3}(y_{1}+y_{2}) 
\text{ and } d'(r_{1, 2})=w'_{2, 3}y_{2} \]
where $w'_{2,3}\in A$ satisfies $2w'_{2, 3}=0$. 
Thus, we obtain 
$H^{1}(\hb{2}; H_A/L_A)\cong 
\{ w'_{2,3}\in A; 2w'_{2,3}=0\}$,
and the universal coefficient theorem implies
$H_1(\hb{2};L)\cong\mathbb{Z}/2\mathbb{Z}$.

Next, we prove that $H_1(\hb{2};H/L)\cong H_1(\hb{2};H)$
and the homomorphism 
$H_2(\hb{2};H/L)\to H_1(\hb{2};L)$ is surjective.
Since $H_0(\hb{2};L)=L_{\hb{2}}=0$ as shown in Lemma~\ref{lem:H_0-L},
we have the exact sequence
\[
\begin{CD}
H_2(\hb{2};H/L)@>>> H_1(\hb{2};L)@>>>H_1(\hb{2};H)@>>>H_1(\hb{2};H/L)@>>>0.
\end{CD}
\]
Thus, it suffices to show that the homomorphism $H_1(\hb{2};L)\to H_1(\hb{2};H)$ 
is the zero map.

As we saw in the proof of Theorem~\ref{thm:main theorem},
$\Ker (f_{2}\colon Z^{1}(\hb{2}; H_{A})\to A^{4})$ 
is contained in the image of the homomorphism
$Z^1(\mathcal{H}_2;L_A)\to Z^1(\mathcal{H}_2;H_A)$.
Thus, the homomorphism
$H^{1}(\hb{2}; H_A)\to H^{1}(\hb{2}; H_A/L_A)$ 
induced by the projection $H_A\to H_A/L_A$ 
is the zero map. 
The universal coefficient theorem implies
$H_1(\hb{2};L)\to H_1(\hb{2};H)$ is the zero map.
\end{proof}

\subsection{$H^{1}(\hb{3}; H_A)$}
Here, we assume $g=3$ and prove that
$H_{1}(\hb{3}; H_A)\cong \Hom(\mathbb{Z}/4\mathbb{Z}\oplus \mathbb{Z}/2\mathbb{Z},A)$.
Then, the universal coefficient theorem implies
$H_1(\hb{3}; H)\cong \mathbb{Z}/4\mathbb{Z}\oplus \mathbb{Z}/2\mathbb{Z}$,
and we complete the proof of Theorem~\ref{thm:main theorem} when $g=3$.

Let $d\in Z^{1}(\hb{3}; H_A)$ be a crossed homomorphism satisfying the condition 
$d(r_{1, 2})_{1}=d(s_{1})_{4}-d(r_{1, 2})_{5}=d(u)_{2}=d(u)_{3}=d(u)_{5}=d(u)_{6}=0$
as in Lemma~\ref{lem:homology-cond}. 
By Lemma~\ref{vanish}, 
we can set 
\begin{align*}
d(a_{1})&=w_{1, 1}x_{1}+w_{1, 2}x_{2}+w_{1, 3}x_{3}, \\
d(s_{1})&=w_{2, 1}x_{1}+w_{2, 2}x_{2}+w_{2, 3}x_{3}
+w_{2, 4}y_{1},  \\
d(t_{1})&=w_{3, 1}x_{1}+w_{3, 2}x_{2}+w_{3, 3}x_{3}
+w_{3, 4}y_{1}+w_{3, 5}y_{2}+w_{3, 6}y_{3}, \\
d(r_{1, 2})&=w_{4, 2}x_{2}+w_{4, 3}x_{3}
+w_{4, 4}y_{1}+w_{4, 5}y_{2}+w_{4, 6}y_{3}. 
\end {align*}
By the condition on $d$ and Lemma~\ref{vanish}, we also have
\begin{equation}
2w_{1, 2}=2w_{1, 3}=0, 
w_{2, 4}=w_{4, 5}=-2w_{1, 1},\text{ and }
w_{1, 2}+w_{4, 4}=0.
\end{equation}

\begin{lemma}\label{relation3-1}
\begin{enumerate}
\item $w_{1, 2}=w_{4, 4}=0$. 
\item $w_{1, 3}=0$. 
\item $2w_{3, 3}=0$. 
\item $w_{3, 4}+w_{3, 5}=0$. 
\item $w_{3, 6}=0$. 
\item $d(d_{1, 2})=w_{1, 2}(x_{1}+x_{2})$. 
\end{enumerate}
\end{lemma}
\begin{proof}
Since $a_{1}$ and $r_{1, 3}$ commute by the relation (P10)(a), 
it must be $(1-a_{1})d(r_{1, 3})=(1-r_{1, 3})d(a_{1})$. 
Since
\[
((1-r_{1, 3})d(a_{1}))_{2}=w_{1, 3},\text{ while }
(1-a_{1})d(r_{1, 3})_{2}=0,
\] 
we have the equation (2). 

Since $d(t_{2})=ud(t_{1})=t_{1}t_{2}d(t_{1})$ 
by Lemma \ref{u-conj}, 
we have 
\begin{align*}
d(t_{2})=(w_{3, 3}+w_{3, 4}+w_{3, 5})x_{1}
+(w_{3, 1}+w_{3, 6})x_{2}
&+(w_{3, 2}+w_{3, 6})x_{3} \\
&+w_{3, 6}y_{1}+w_{3, 4}y_{2}+w_{3, 5}y_{3}. 
\end{align*}
Since $a_{1}$ and $t_{2}$ commute by the relation (P7), 
it must be $(1-a_{1})d(t_{2})=(1-t_{2})d(a_{1})$. 
Since
\[
((1-t_{2})d(a_{1}))_{2}=w_{1, 2} \text{ while }(1-a_{1})d(t_{2})_{2}=0,
\] we have 
the equation (1). 
Furthermore, since 
\[
(1-a_{1})d(t_{2})_{1}=-w_{3, 6} \text{ while }((1-t_{2})d(a_{1}))_{1}=0,
\]
we have the equation (5). 

Now we have $d(a_{1})=w_{1, 1}x_{1}$. 
Note that by Lemma \ref{u-conj},
$d(a_{i})=u^{i-1}d(a_{1})=w_{1, 1}x_{i}$ for any $1\leq i\leq 3$. 
Since $d_{1, 2}=r_{1, 2}a_{2}r_{1, 2}^{-1}$ by the relation (P10)(a), 
we have 
\[ d(d_{1, 2})=(1-d_{1, 2})d(r_{1, 2})+r_{1, 2}d(a_{2})
=(-w_{1, 1}-w_{4, 5})(x_{1}+x_{2}). \] 
Since $w_{4, 5}=-2w_{1, 1}$, 
we have the equation (6). 

Since $d_{1, 2}$ and $t_{1}$ commute by the relation (P8), 
it must be $(1-d_{1, 2})d(t_{1})=(1-t_{1})d(d_{1, 2})=0$. 
Since $(1-d_{1, 2})d(t_{1})=-(w_{3, 4}+w_{3, 5})(x_{1}+x_{2})$, 
we have the equation (4). 

Since $s_{1}$ and $t_{2}$ commute by the relation (P7), 
it must be $(1-s_{1})d(t_{2})=(1-t_{2})d(s_{1})$. 
Since $((1-s_{1})d(t_{2}))_{1}=2w_{3, 3}$
while $((1-t_{2})d(t_{1}))_{1}=0$, 
we have the equation (3). 
\end{proof}
Since $((1-t_{2})d(s_{1}))_{2}=w_{2, 2}-w_{2, 3}-w_{3, 4}$ 
while $(1-s_{1})d(t_{2})_{2}=0$, 
we have 
\begin{align}\label{relation3-A}
w_{2, 2}-w_{2, 3}-w_{3, 4}=0. 
\end{align}
\begin{lemma}\label{relation3-2}
\[ w_{3, 2}=w_{3, 3}
\text{ and } w_{3, 4}=w_{3, 5}=0. \]
\end{lemma}
\begin{proof}
Since $d(u)=d(t_{1}t_{2})=d(t_{1})+t_{1}d(t_{2})$, 
a straightforward computation shows
\[
d(u)_{2}=w_{3, 2}+w_{3, 3}+w_{3, 4}, 
d(u)_{3}=w_{3, 2}+w_{3, 3},
\text{ and }d(u)_{5}=d(u)_{6}=w_{3, 5}.
\]
Since $d(u)_{2}, d(u)_{3}, d(u)_{5}, d(u)_{6}=0$, 
we have 
\[ w_{3, 2}+w_{3, 3}=w_{3, 4}=w_{3, 5}=0. \]
Lemma~\ref{relation3-1}~(3) implies $w_{3, 2}=w_{3, 3}$.
\end{proof}

\begin{lemma}\label{relation3-4}
\[
w_{2, 2}=w_{2, 3},\quad
2w_{2, 2}=2w_{2, 3}=0.
\]
\end{lemma}

\begin{proof}
The equation $w_{2, 2}=w_{2, 3}$ follows from Lemma \ref{relation3-2} and Equation (\ref{relation3-A}).

Next, we prove $2w_{2, 2}=2w_{2, 3}=0$.
Since $s_{1}$ and $r_{-1, 2}$ commute by the relation (P10)(c), 
it must be $(1-s_{1})d(r_{-1, 2})=(1-r_{-1, 2})d(s_{1})$. 
Since 
\[
((1-r_{-1, 2})d(s_{1}))_{2}=2w_{2, 2}\text{ while }
(1-s_{1})d(r_{-1, 2})_{2}=0,
\]
we obtain $2w_{2, 2}=2w_{2, 3}=0$.
\end{proof}
\begin{lemma}\label{relation3-5}
\[ 4w_{1, 1}=2w_{2, 4}=2w_{3, 1}=2w_{4, 5}=0. \]
\end{lemma}

\begin{proof}
As in (P8), $t_{1}$ and $s_{1}t_{1}s_{1}$ commute.
Thus it must be $(1-t_{1})d(s_{1}t_{1}s_{1})=(1-s_{1}t_{1}s_{1})d(t_{1})$. 
Since 
\begin{align*}
d(s_{1}t_{1}s_{1})
&=d(s_{1})+s_{1}d(t_{1})+s_{1}t_{1}d(s_{1}) \\
&=(w_{2, 1}+w_{2, 2}-w_{2, 4}-w_{3, 1})x_{1}
+(w_{2, 1}+w_{2, 2}+w_{3, 2})x_{2}+w_{3, 3}x_{3}+w_{2, 4}(y_{1}+y_{2}), 
\end{align*}
we have 
\begin{align*}
(1-t_{1})d(s_{1}t_{1}s_{1})
&=-(w_{3, 1}+w_{3, 2})(x_{1}-x_{2})-2w_{2, 4}x_{1},\text{ while }\\
(1-s_{1}t_{1}s_{1})d(t_{1})
&=(w_{3, 1}-w_{3, 2})(x_{1}-x_{2}).
\end{align*}
Hence we have $2w_{2, 4}=2w_{3, 1}=0$.
Since $w_{2, 4}=w_{4, 5}=-2w_{1, 1}$, 
we also have $4w_{1, 1}=2w_{2, 4}=2w_{4, 5}=0$.
\end{proof}
\begin{lemma}\label{relation3-6}
\[ w_{4, 2}=2w_{4, 3}=0. \]
\end{lemma}

\begin{proof}
By the relation $r_{1, 2}^{2}=s_{2}d_{1, 2}s_{2}d_{1, 2}^{-1}$ in (P9),
we have $d(r_{1, 2}^{2})=d(s_{2}d_{1, 2}s_{2}d_{1, 2}^{-1})$.
First, we have
\[
d(r_{1, 2}^{2})
=d(r_{1, 2})+r_{1, 2}d(r_{1, 2}) 
=(-w_{4, 2}+w_{4 ,5})x_{1}+2w_{4, 3}x_{3}+2w_{4, 6}y_{3}.
\]
Next, we compute $d(s_{2}d_{1, 2}s_{2}d_{1, 2}^{-1})$.
Since $d_{1, 3}=t_{2}d_{1, 2}t_{2}^{-1}$ 
and $d_{2, 3}=t_{1}d_{1, 3}t_{1}^{-1}$, 
we have
\[
d(d_{1, 3})=w_{1, 1}(x_{1}+x_{3}) 
\text{ and }d(d_{2, 3})=w_{1, 1}(x_{2}+x_{3}).
\]
Recall that $k_{i}=a_{i}a_{i+1}t_{i}d_{i, i+1}^{-1}$ for $i=1, 2$. 
Hence we have 
\begin{align*}
d(k_{i})&=d(a_{i}a_{i+1}t_{i}d_{i, i+1}^{-1}) \\
&=d(a_{i})+a_{i}d(a_{i+1})+a_{i}a_{i+1}d(t_{i})
-k_{i}d(d_{i, i+1}) \\
&=d(t_{i}). 
\end{align*}
Since $s_{i+1}=k_{i}s_{i}k_{i}^{-1}$ for $i=1, 2$, 
\[ d(s_{i+1})=d(k_{i})+k_{i}d(s_{i})-s_{i+1}d(k_{i})
=k_{i}d(s_{i}). \]
Thus we obtain
\begin{align*}
d(s_{2}d_{1, 2}s_{2}d_{1, 2}^{-1})
&=(1+s_{2}d_{1, 2})d(s_{2})
+s_{2}(1-d_{1, 2}s_{2}d_{1, 2}^{-1})d(d_{1, 2}) \\
&=-2w_{1, 1}x_{2}+w_{2, 4}(x_{1}+x_{2}) \\
&=w_{2, 4}x_{1}.
\end{align*}
Comparing $d(r_{1, 2}^{2})$ and $d(s_{2}d_{1, 2}s_{2}d_{1, 2}^{-1})$,
we have $w_{4, 2}=2w_{4, 3}=0$.
\end{proof}
\begin{lemma}\label{relation3-7}
\[ w_{3, 1}+w_{3, 2}=w_{4, 5}, w_{3, 3}=w_{4, 3} 
\text{ and }w_{4, 6}=0. \]
In particular, 
\[ 2w_{3, 1}=2w_{3, 2}=2w_{3, 3}=2w_{4, 3}=0. \]
\end{lemma}

\begin{proof}
The relation $t_{1}r_{1, 2}t_{1}=r_{1, 2}t_{1}r_{1, 2}$ in (P11) shows
\[
(1-t_{1}+r_{1, 2}t_{1})d(r_{1, 2})=(1-r_{1, 2}+t_{1}r_{1, 2})d(t_{1}).
\] 
A straightforward calculation shows
\begin{align*}
(1-r_{1, 2}+t_{1}r_{1, 2})d(t_{1})
&=(w_{3, 1}+w_{3, 2})x_{2}+w_{3, 3}x_{3},\text{ and }\\
(1-t_{1}+r_{1, 2}t_{1})d(r_{1, 2})
&=w_{4, 5}x_{2}+w_{4, 3}x_{3}+w_{4, 6}y_{3}.
\end{align*}
Thus we have $w_{3, 1}+w_{3, 2}=w_{4, 5}$, $w_{3, 3}=w_{4, 3}$, and $w_{4, 6}=0$.
\end{proof}
\begin{lemma}\label{relation3-8}
\[ w_{2, 1}=0. \]
\end{lemma}
\begin{proof}
As in (P12), $r_{1, 3}=s_{3}c_{1, 3}s_{3}c_{1, 3}^{-1}
k_{2}a_{3}a_{1}t_{2}d_{1, 2}^{-1}t_{2}^{-1}r_{1, 2}^{-1}s_{2}h_{2}r_{1, 2}^{-1}h_{2}^{-1}k_{2}^{-1}$, 
where $h_{2}=k_{2}^{-1}t_{1}^{-1}k_{2}$. 
Since 
\begin{align}
d(h_{2})&=w_{3, 1}x_{1}+w_{3, 2}x_{2}+w_{3, 3}x_{3} \text{ and }\notag\\
d(s_{3}c_{1, 3}s_{3}c_{1, 3}^{-1})&=w_{2, 4}(x_{1}+x_{2}), \label{dscsc}
\end{align}
we have 
\begin{align*}
d(r_{1, 3})&=(-w_{2, 1}+w_{2, 2})x_{1}+(w_{2, 3}+w_{3, 2})x_{2}-w_{2, 1}x_{3}
+w_{4, 5}y_{3} \text{ and} \\
d(r_{1, 3}^{2})
&=d(r_{1, 3})+r_{1, 3}d(r_{1, 3}) 
=(-w_{2, 1}+w_{4 ,5})x_{1}+(w_{2, 1}+w_{4 ,5})x_{2} 
\end{align*}
by a straightforward calculation. 
The relation $r_{1, 3}^{2}=s_{3}c_{1, 3}s_{3}c_{1, 2}^{-1}$ in (P9)
and the equations (\ref{dscsc}) show 
that $w_{2, 1}=0$.
\end{proof}

\begin{lemma}\label{relation3-9}
\[ w_{3, 2}=w_{3, 3}=w_{4, 3}=0. \]
In particular, 
\[ w_{2, 4}=w_{3, 1}=w_{4, 5}=2w_{1, 1}. \]
\end{lemma}

\begin{proof}
It is sufficient to prove that $w_{3, 2}=0$. 
Recall that $z=(a_{1}a_{2}a_{3})s_{1}t_{1}t_{2}s_{1}t_{1}s_{1}c_{1, 3}$ 
and $z_{3}=k_{2}k_{1}z$. 
Hence we have 
\[ d(z)=w_{2, 4}(x_{1}+x_{2}+x_{3}+y_{1}+y_{2}+y_{3})
+w_{3, 3}x_{1}+w_{3, 1}x_{2}+w_{3, 2}x_{2},\]
and $d(z_{3})=d(z)$. 
Since $r_{1, 3}$ and $z_{3}$ commute by the relation (P10)(e), 
it must be 
$(1-r_{1, 3})d(z_{3})=(1-z_{3})d(r_{1, 3})$. 
A straightforward calculation shows
\[ (1-r_{1, 3})d(z_{3})=w_{3, 2}(x_{1}+x_{2}) \]
while $(1-z_{3})d(r_{1, 3})=0$. 
Thus we obtain $w_{3, 2}=0$.
\end{proof}
Summarizing Lemmas \ref{relation3-1}, 
\ref{relation3-2}, \ref{relation3-4}, \ref{relation3-5},
\ref{relation3-6}, \ref{relation3-7}, \ref{relation3-8}
and \ref{relation3-9},
we have 
\begin{align*}
d(a_{1})=w_{1, 1}x_{1}, d(s_{1})=w_{2, 2}(x_{2}+x_{3})+2w_{1, 1}y_{1}, \\
d(t_{1})=2w_{1, 1}x_{1}, \text{ and } d(r_{1, 2})=2w_{1, 1}y_{2} 
\end{align*}
where $4w_{1, 1}=2w_{2, 2}=0$. 
It can be verified that such $d$ 
is compatible with the relations (P1)--(P12).
Now we have 
\begin{align*}
H^{1}(\hb{3}; H_A)\cong 
\Ker f_{3}\cong
\{(w_{1, 1}, w_{2, 2})\in A^{2}; 4w_{1, 1}=2w_{2, 2}=0 \} .
\end{align*}
\begin{prop}\label{prop:H/L genus3}
\[
H_1(\hb{3};H/L)\cong (\mathbb{Z}/2\mathbb{Z})^2,
\]
and the image of the homomorphism $H_2(\hb{3};H/L)\to H_1(\hb{3};L)$ induced by
the exact sequence $0\to L\to H\to H/L\to 0$ is isomorphic to $\mathbb{Z}/2\mathbb{Z}$.
\end{prop}
\begin{proof}
As we saw in the proof of Theorem~\ref{thm:main theorem},
under the isomorphism 
\[
H^1(\hb{3};H_A)\cong \{(w_{1,1}, w_{2,2})\in A^2; 4w_{1,1}=2w_{2,2}=0\},
\]
the submodule $\{(w_{1,1}, w_{2,2})\in A^2; 2w_{1,1}=2w_{2,2}=0\}$
is in $\Im(H^{1}(\hb{3}; L_A)\to H^{1}(\hb{3}; H_A))$.
The universal coefficient theorem implies
$\Im(H_1(\hb{3};H)\to H_1(\hb{3};H/L))$ is of order at least $4$.
On the other hand, $H_1(\hb{3};H/L)$ is at most order $4$ as explained in Remark~\ref{rem:genus23H/L}.
Thus we obtain $H_1(\hb{3};H/L)\cong (\mathbb{Z}/2\mathbb{Z})^2$.

By Lemma~\ref{lem:H_0-L}, the coinvariant $H_0(\hb{3};L)$ is trivial.
Thus the homomorphism $H_1(\hb{3};H)\to H_1(\hb{3};H/L)$ is surjective.
Since $H_1(\hb{3};H)\cong\mathbb{Z}/4\mathbb{Z}\oplus \mathbb{Z}/2\mathbb{Z}$, 
we have 
$H_1(\hb{3};L)\cong (\mathbb{Z}/2\mathbb{Z})^2$ 
and $H_1(\hb{3};H/L)\cong (\mathbb{Z}/2\mathbb{Z})^2$. 
The exact sequence
\[
\begin{CD}
H_2(\hb{3};H/L)@>>>H_1(\hb{3};L)@>>>H_1(\hb{3};H)@>>>H_1(\hb{3};H/L)@>>>0
\end{CD}
\]
shows $\Im(H_2(\hb{3};H/L)\to H_1(\hb{3};L))\cong \mathbb{Z}/2\mathbb{Z}$.
\end{proof}

\section{Proof of Theorem~\ref{thm:main theorem2}}
In this section, we prove Theorem~\ref{thm:main theorem2},
and calculate $H_1(\mathcal{H}_g^*;L)$ and $H_1(\mathcal{H}_g^*;H/L)$.

\begin{lemma}\label{lem:coinvariant}
\[
(L\otimes L^*)_{\mathcal{H}_g}\cong (L\otimes H)_{\mathcal{H}_g}\cong\mathbb{Z}.
\]
\end{lemma}

\begin{proof}
The action of $\mathcal{H}_g$ on $L$ factors through $\GL(g;\mathbb{Z})$,
and $L$ is isomorphic to $V=\mathbb{Z}^g$ endowed with the natural left $\GL(g;\mathbb{Z})$-module.
Thus, the fact that $(V\otimes V^*)_{\GL(g;\mathbb{Z})}\cong \mathbb{Z}$ implies $(L\otimes L^*)_{\mathcal{H}_g}=\mathbb{Z}$.

Next, recall that the intersection form on $H$ induces an isomorphism $L^*\cong H/L$.
Since there is an exact sequence 
\[
\begin{CD}
(L^{\otimes2})_{\mathcal{H}_g}@>>> (L\otimes H)_{\mathcal{H}_g}@>>> (L\otimes L^*)_{\mathcal{H}_g}@>>>0,
\end{CD}
\]
it suffices to show that $\Im((L^{\otimes 2})_{\mathcal{H}_g}\to (L\otimes H)_{\mathcal{H}_g})$ is trivial.
Since this image is generated by $x_i\otimes x_j$ for $1\le i\le g$ and $1\le j\le g$, and
\[
a_j(x_i\otimes y_j)-x_i\otimes y_j=x_i\otimes x_j,
\]
we obtain $\Im((L^{\otimes 2})_{\mathcal{H}_g}\to (L\otimes H)_{\mathcal{H}_g})=0$.
\end{proof}

\begin{proof}[Proof of Theorem~\ref{thm:main theorem2}]
The exact sequence
$1\to \mathbb{Z}\to \mathcal{H}_{g,1}\to\mathcal{H}_g^*\to 1$
induces an exact sequence
\[
\begin{CD}
H_1(\mathbb{Z};H)_{\mathcal{H}_g^*}@>>>H_1(\mathcal{H}_{g,1};H)@>>>H_1(\mathcal{H}_g^*;H)@>>>0.
\end{CD}
\]
Since $H_1(\mathbb{Z};H)_{\mathcal{H}_g^*}\cong H_{\mathcal{H}_g^*}=0$,
we obtain an isomorphism $H_1(\mathcal{H}_{g,1};H)\cong H_1(\mathcal{H}_g^*;H)$.

Next, we compute $H_1(\mathcal{H}_g^*;H)$.
Morita showed that $H_1(\mathcal{M}_g^*;H)\cong \mathbb{Z}$ 
and the forgetful exact sequence
$1\to\pi_1\Sigma_g\to \mathcal{M}_g^*\to\mathcal{M}_g\to 1$ induces 
an exact sequence
\[
\begin{CD}
0@>>> \mathbb{Z}@>>> H_1(\mathcal{M}_g^*;H)@>>> H_1(\mathcal{M}_g;H)@>>>0
\end{CD}
\]
when $g\ge2$. 
In Lemma~\ref{lem:coinvariant}, we showed the isomorphism $H_1(\pi_1\Sigma_g;H)_{\mathcal{H}_g^*}\cong \mathbb{Z}$ induced by the intersection form on $H$.
Thus, restricting the exact sequence to $\mathcal{H}_g^*$,
we obtain a commutative diagram
\[
\begin{CD}
@.\mathbb{Z}@>>> H_1(\mathcal{H}_g^*;H)@>>> H_1(\mathcal{H}_g;H)@>>>0\\
@.@|@VVV@VVV\\
0@>>> \mathbb{Z}@>>> H_1(\mathcal{M}_g^*;H)@>>> H_1(\mathcal{M}_g;H)@>>>0.
\end{CD}
\]
By the above diagram,
both the kernels and the cokernels of the homomorphisms $H_1(\mathcal{H}_g^*;H)\to H_1(\mathcal{M}_g^*;H)$ and $H_1(\mathcal{H}_g;H)\to H_1(\mathcal{M}_g;H)$ coincide.
By Remark~\ref{rem:surjective-homo} and Theorem~\ref{thm:main theorem},
we see that $\Coker(H_1(\mathcal{H}_g^*;H)\to H_1(\mathcal{M}_g^*;H))$ 
is trivial for $g\ge2$ 
and that $\Ker(H_1(\mathcal{H}_g^*;H)\to H_1(\mathcal{M}_g^*;H))$ 
is trivial when $g\ge4$ 
and is isomorphic to $\mathbb{Z}/2\mathbb{Z}$ when $g=2,3$.
Thus we can determine $H_1(\mathcal{H}_g^*;H)$.
\end{proof}

\begin{prop}
\begin{enumerate}
\item
When $g\ge2$, the forgetful homomorphism
$\mathcal{H}_g^*\to \mathcal{H}_g$
induces an isomorphism
\[
H_1(\mathcal{H}_g^*;H/L)\cong H_1(\mathcal{H}_g;H/L).
\]
In particular, we have
\[
H_1(\mathcal{H}_g^*;H/L)\cong
\begin{cases}
\mathbb{Z}/2\mathbb{Z}&\text{if }g\ge4,\\
(\mathbb{Z}/2\mathbb{Z})^2&\text{if }g=2,3.
\end{cases}
\]
\item 
When $g\ge4$,
the homomorphism $H_1(\mathcal{H}_g^*;L)\to H_1(\mathcal{H}_g^*;H)$
induced by the inclusion $L\to H$ is injective.
When $g=2,3$,
$\Ker(H_1(\mathcal{H}_g^*;L)\to H_1(\mathcal{H}_g^*;H))\cong \mathbb{Z}/2\mathbb{Z}$.
In particular, we have
\[
H_1(\mathcal{H}_g^*;L)\cong
\begin{cases}
\mathbb{Z}&\text{if }g\ge4,\\
\mathbb{Z}\oplus\mathbb{Z}/2\mathbb{Z}&\text{if }g=2,3.
\end{cases}
\]
\end{enumerate}
\end{prop}

\begin{proof}
Consider the exact sequences between homology groups with coefficients in $L$
induced by the forgetful exact sequences
$1\to\pi_1\Sigma_g\to \mathcal{M}_g^*\to\mathcal{M}_g\to 1$
and its restriction $1\to\pi_1\Sigma_g\to \mathcal{H}_g^*\to\mathcal{H}_g\to 1$.
Applying Lemma~\ref{lem:coinvariant}, we obtain a commutative diagram
\[
\begin{CD}
@.\mathbb{Z}@>>>H_1(\mathcal{H}_g^*;L)@>>>H_1(\mathcal{H}_g;L)@>>>0\\
@.@|@VVV@VVV\\
0@>>>\mathbb{Z}@>>>H_1(\mathcal{H}_g^*;H)@>>>H_1(\mathcal{H}_g;H)@>>>0.
\end{CD}
\]
By the above diagram, 
both of the kernels and the cokernels of the homomorphisms
$H_1(\mathcal{H}_g^*;L)\to H_1(\mathcal{H}_g^*;H)$
and $H_1(\mathcal{H}_g;L)\to H_1(\mathcal{H}_g;H)$
coincide.
In particular, by Lemma~\ref{lem:H_0-L}, we obtain
\begin{align*}
H_1(\mathcal{H}_g^*;H/L)
&\cong \Coker(H_1(\mathcal{H}_g^*;L)\to H_1(\mathcal{H}_g^*;H))\\
&\cong \Coker(H_1(\mathcal{H}_g;L)\to H_1(\mathcal{H}_g;H))\\
&\cong H_1(\mathcal{H}_g;H/L).
\end{align*}
In Remark~\ref{rem:L-H} and Propositions~\ref{prop:H/L genus2} and \ref{prop:H/L genus3},
We see that
$\Ker(H_1(\mathcal{H}_g;L)\to H_1(\mathcal{H}_g;H))$ is trivial when $g\ge4$,
and is isomorphic to $\mathbb{Z}/2\mathbb{Z}$.
In Lemma~\ref{lem:first homology-H} and Propositions~\ref{prop:H/L genus2} and \ref{prop:H/L genus3},
we also see that
$\Coker(H_1(\mathcal{H}_g;L)\to H_1(\mathcal{H}_g;H))$
is isomorphic to $\mathbb{Z}/2\mathbb{Z}$ when $g\ge4$,
and is isomorphic to $(\mathbb{Z}/2\mathbb{Z})^2$ when $g=2,3$.
Thus, we can determine $H_1(\mathcal{H}_g^*;L)$.
\end{proof}

\vskip 5pt

\noindent \textbf{Acknowledgments.}
The authors wish to express their gratitude to
Susumu Hirose for his helpful advices. 
The first-named
author is supported by JSPS Research Fellowships
for Young Scientists (26$\cdot$110).


\providecommand{\bysame}{\leavevmode\hbox to3em{\hrulefill}\thinspace}
\providecommand{\MR}{\relax\ifhmode\unskip\space\fi MR }
\providecommand{\MRhref}[2]{%
  \href{http://www.ams.org/mathscinet-getitem?mr=#1}{#2}
}
\providecommand{\href}[2]{#2}


\begin{thebibliography}{1}
\bibitem{benson91}
D.~J.~Benson and F.~R.~Cohen, \emph{The mod-2 cohomology of the mapping class group for a surface of genus 2},
Mem. Amer. Math. Soc. \textbf{443} (1991), 93--104. \MR{}


\bibitem{evens91}
L.~Evens, \emph{The cohomology of groups},
Oxford Mathematical Monographs,
Clarendon Press, 1991. \MR{1144017 (93i:20059)}

\bibitem{harer83}
J.~Harer, \emph{The second homology group of the mapping class group of an orientable surface},
Invent. Math. \textbf{72} (1983) no.~2, 221--239. \MR{700769 (84g:57006)}

\bibitem{harer85}
J.~Harer, \emph{Stability of the homology of the mapping class groups of orientable surfaces},
Annals of Mathematics \textbf{121} (1985), 215--249. \MR{830043  (87c:32030)}

\bibitem{hatcher10}
A.~Hatcher and N.~Wahl, \emph{Stabilization for mapping class groups of 3-manifolds},
Duke Math. J. \textbf{155} (2010), no.~2, 205--269. \MR{2736166 (2012c:57001)}

\bibitem{johnson80}
D.~Johnson, \emph{Spin structures and quadratic forms on surfaces},
J. London Math. Soc. \textbf{22} (1980), no.~2, 365--373. \MR{588283 (81m:57015)}

\bibitem{kawazumi01}
N.~Kawazumi and S.~Morita, \emph{The primary approximation to the cohomology of the moduli space of curves and cocycles for the Mumford-Morita-Miller classes}, preprint Univ. of Tokyo, 2001.

\bibitem{kawazumi05}
N.~Kawazumi, \emph{Cohomological aspects of Magnus expansions},
preprint, arXiv:0505497.

\bibitem{korkmaz03}
M.~Korkmaz and A.~I.~Stipsicz, \emph{The second homology groups of mapping class groups of oriented surfaces},
Math. Proc. Cambridge Philos. Soc. \textbf{134} (2003), no.~3, 479--489. \MR{1981213 (2004c:57033)}

\bibitem{luft78}
E.~Luft, \emph{Actions of the homeotopy group of an orientable {$3$}-dimensional handlebody},
Math. Ann. \textbf{234} (1978), no.~3, 279--292. \MR{0500977 (58 \#18461)}

\bibitem{morita87}
S.~Morita, \emph{Characteristic classes of surface bundles},
Invent. Math. \textbf{90} (1987), no.~3, 551--577. \MR{0914849  (89e:57022)}

\bibitem{morita89}
S.~Morita, \emph{Families of jacobian manifolds and characteristic classes of surface bundles. I},
Ann. Inst. Fourier \textbf{39} (1989), no.~3, 777--810. \MR{1030850 (91d:32028)}

\bibitem{morita93}
S.~Morita, \emph{The extension of Johnson's homomorphism from the Torelli group to the mapping class group},
Invent. math. \textbf{111} (1993), 197--224. \MR{1193604 (93j:57001)}

\bibitem{pitsch13}
W.~Pitsch, \emph{The 2-torsion in the second homology of the genus 3 mapping class group},
preprint, arXiv:1311.5705.

\bibitem{popescu11}
C.~R. Popescu, \emph{A simple presentation of the handlebody group of genus 2},
Bull. Math. Soc. Sci. Math. Roumanie (N.S.) \textbf{54(102)} (2011), no.~1, 83--92. \MR{2799246 (2012f:57004)}

\bibitem{sakasai12}
T.~Sakasai, \emph{Lagrangian mapping class groups from a group homological point of view},
Algebr. Geom. Topol. \textbf{12} (2012), no,~1, 267--291. \MR{2916276}

\bibitem{satoh06}
T.~Satoh, \emph{Twisted first homology groups of the automorphism group of a free group},
J. Pure Appl. Algebra \textbf{204} (2006), no.~2, 334--348. \MR{2184815 (2006j:20079)}

\bibitem{stukow15}
M.~Stukow, \emph{The first homology group of the mapping class group of a nonorientable surface with twisted coefficients},
preprint, arXiv:1501.01810.

\bibitem{wajnryb98}
B.~Wajnryb, \emph{Mapping class group of a handlebody},
Fund. Math. \textbf{158} (1998), no.~3, 195--228. \MR{1663329 (2000a:20075)}
\end{thebibliography}
\end{document}